\documentclass[11pt, a4paper, oneside, DIV11, final, pagebackref]{amsart}
\usepackage[notref,notcite]{showkeys}
\usepackage[latin1]{inputenc}
\usepackage[T1]{fontenc}
\usepackage{amssymb}
\usepackage[stretch=10]{microtype}
\usepackage{fixltx2e}
\usepackage{fix-cm}
\usepackage{mathtools}
\usepackage[DIV11, headinclude=true]{typearea}
\usepackage[hidelinks, linktoc=all]{hyperref}

\providecommand{\href}[2]{#2}
\providecommand{\texorpdfstring}[2]{#1}
\providecommand*{\backref}{}
\providecommand*{\backrefalt}{}
\renewcommand*{\backref}[1]{}
\renewcommand*{\backrefalt}[4]{%
	\ifcase #1 %
	\or
	  Cited page~#2.
	\else
	  Cited pages~#2.
	\fi
}

\newcommand{\Pbb}{\mathbb{P}}
\newcommand{\given}{\mid}
\newcommand{\E}{\mathbb{E}}
\newcommand{\Z}{\mathbb{Z}}
\newcommand{\N}{\mathbb{N}}
\newcommand{\R}{\mathbb{R}}

\newcommand{\boF}{\mathcal{F}}
\newcommand{\boL}{\mathcal{L}}
\newcommand{\dd}{\mathop{}\!\mathrm{d}}
\DeclarePairedDelimiter{\abs}{\lvert}{\rvert}
\DeclarePairedDelimiter{\norm}{\lVert}{\rVert}
\newcommand{\st}{\::\:}
\DeclareMathOperator{\Lip}{Lip}
\DeclareMathOperator{\Card}{Card}
\DeclareMathOperator{\tail}{tail}
\newcommand{\restr}[1]{_{\lvert #1}}

\DeclareMathOperator{\sgn}{sgn}

\newcommand{\coloneqq}{\mathrel{\mathop:}=}
\renewcommand{\epsilon}{\varepsilon}

\renewcommand{\phi}{\varphi}
\renewcommand{\leq}{\leqslant}
\renewcommand{\geq}{\geqslant}
\newcommand{\ic}{\mathbf{i}}

\newtheorem{thm}{Theorem}[section]
\newtheorem{prop}[thm]{Proposition}

\newtheorem{lem}[thm]{Lemma}
\newtheorem{cor}[thm]{Corollary}
\newtheorem*{prop*}{Proposition}

\theoremstyle{definition}

\newtheorem{rmk}[thm]{Remark}

\numberwithin{equation}{section}

\title[Moment bounds and concentration inequalities]{Moment bounds and concentration inequalities for slowly mixing dynamical systems}
\author{S\'ebastien Gou\"ezel and Ian Melbourne}

\address{IRMAR, CNRS UMR 6625,
Universit\'e de Rennes 1, 35042 Rennes, France}
\email{sebastien.gouezel@univ-rennes1.fr}

\address{Mathematics Institute, University of Warwick, Coventry, CV4 7AL, UK}
\email{I.Melbourne@warwick.ac.uk}
\thanks{The research of IM was supported in part by the European Advanced Grant StochExtHomog (ERC AdG 320977).}

\date{September 16, 2014}

\begin{document}

\begin{abstract}
We obtain optimal moment bounds for Birkhoff sums, and optimal concentration
inequalities, for a large class of slowly mixing dynamical systems, including
those that admit anomalous diffusion in the form of a stable law or a
central limit theorem with nonstandard scaling $(n\log n)^{1/2}$.
\end{abstract}

\maketitle

\section{Statement of results}

Consider a dynamical system $T$ on a space $X$, preserving an ergodic
probability measure $\mu$. If $x$ is distributed according to $\mu$,
the process $x, Tx, T^2 x,\dotsc$ on $X^\N$ is stationary, with
distribution $\mu\otimes \delta_{Tx} \otimes \delta_{T^2 x}\otimes
\dotsm$ (equivalently, one considers a Markov chain on $X$, with
stationary measure $\mu$, for which the transitions from $x$ to $Tx$ are
deterministic). In particular, if $f$ is a
real-valued function on $X$, the real process $f(x), f(Tx),\dotsc$ is
also stationary. We would like to understand to what extent these
processes behave like independent or weakly dependent processes:
Although they are purely deterministic once the starting point $x$ is
fixed, one expects a random-like behaviour if the map $T$ is
sufficiently chaotic and the observable $f$ is regular enough. In
such a situation, the Birkhoff sums $S_n f=\sum_{i=0}^{n-1}f\circ
T^i$ of H\"older continuous functions with zero average typically
satisfy the central limit theorem, and grow like $\sqrt{n}$. On the
other hand, the moments $\int \abs{S_n f}^p \dd\mu$ may grow faster
than $n^{p/2}$: it is possible that some subsets of $X$ with small
measure give a dominating contribution to those moments. Estimating
the precise growth rate is important from the point of view of large
deviations. It turns out that this precise growth rate depends on
finer characteristics of the system, and displays a transition at
some critical exponent $p^*$ directly related to the lack of uniform
hyperbolicity of the system.

The situation for uniformly expanding/hyperbolic (Axiom~A) systems is
easily described: all moments $\int \abs{S_n f}^p\dd\mu$ grow like
$n^{p/2}$ and moreover  $\int \abs{n^{-1/2}S_n f}^p\dd\mu$ converges
to the $p$'th moment of the limiting Gaussian in the central limit
theorem. \cite{MTorok12} showed that convergence of all moments holds
also for nonuniformly expanding/hyperbolic diffeomorphisms modelled
by Young towers with exponential tails~\cite{lsyoung_annals}.
However, it follows
from~\cite{melbourne_nicol_large_deviations,MTorok12} that the
situation is quite different for systems modelled by Young towers
with polynomial tails~\cite{lsyoung_recurrence}.

In this paper, we give optimal bounds for all moments of Birkhoff
sums (by optimal, we mean that we have upper and lower bounds of the
same order of magnitude), in the situation of Young towers. Many real
systems are quotients of such Young towers, hence our bounds apply to
such systems, including notably intermittent maps of the
interval~\cite{liverani_saussol_vaienti,pomeau_manneville}. See for
instance~\cite{melbourne_nicol_large_deviations} for a discussion of
such applications. Our techniques also give a generalization of
moment inequalities, to concentration inequalities
(see~\cite{gouezel_chazottes_concentration} for a discussion of
numerous applications of such bounds). By the methods
in~\cite{gouezel_chazottes_concentration,MTorok12}, all results
described here pass over to the situation of invertible systems and
flows; for brevity we present the results only for noninvertible
discrete time dynamics.

\medskip

We formulate our results in the abstract setting of Young
towers. To illustrate this setting, let us start with a more concrete
example, intermittent maps, i.e., maps of the interval which are
uniformly expanding away from an indifferent fixed point
(see~\cite{lsyoung_recurrence} for more details). For $\gamma \in
(0,1)$, consider for instance the corresponding
Liverani-Saussol-Vaienti map $T_\gamma:[0,1]\to [0,1]$ given by
  \begin{equation*}
  T_\gamma(x) =
  \begin{cases} x(1+2^\gamma x^\gamma) & \text{ if } x<1/2;
  \\
  2x-1 &\text{ if }x\geq 1/2.
  \end{cases}
  \end{equation*}
The first return map to the subinterval $Y=[1/2,1]$ is uniformly
expanding and Markov. Define a new space $X=\{(x,i) \st x\in [1/2,1],
i<\phi(x)\}$ where $\phi:Y\to \N^*$ is the first return time to $Y$,
i.e., $\phi(x)=\inf\{i>0 \st T_\gamma^i(x) \in Y\}$. On this new
space, we define a dynamics by $T(x,i)=(x,i+1)$ if $i+1<\phi(x)$, and
$T(x,\phi(x)-1)=(T_\gamma^{\phi(x)}(x),0)$. We think of $X$ as a
tower, where the dynamics $T$ is trivial when one climbs up while it
has a large expansion when one comes back to the bottom of the tower.
The point of this construction is that the combinatorics of $T$ are
simpler than those of the original map $T_\gamma$, while the
essential features of $T$ and $T_\gamma$ are the same. More
precisely, the two maps are semiconjugate: the projection $\pi:X\to
[0,1]$ given by $\pi(x,i)=T_\gamma^i(x)$ satisfies $T_\gamma\circ
\pi=\pi \circ T$. Hence, results for the decay of correlations, or
growth of moments, or concentration, for $T$ readily imply
corresponding results for $T_\gamma$. This situation is not specific
to the maps $T_\gamma$: many concrete maps can be modelled by
Young towers in the same way (although the Young tower is usually not
as explicit as in this particular example).

\medskip

Let us give a more formal definition. A Young tower is a space $X$
endowed with a partition $\bigcup_{\alpha} \bigcup_{0\leq i <
h_\alpha} \Delta_{\alpha,i}$ (where $\alpha$ belongs to some
countable set, and $h_\alpha$ are positive integers), a probability
measure $\mu$ and a map $T$ preserving $\mu$. The dynamics $T$ maps
bijectively $\Delta_{\alpha,i}$ to $\Delta_{\alpha, i+1}$ for
$i+1<h_\alpha$, and $\Delta_{\alpha, h_\alpha-1}$ to $\Delta_0
=\bigcup \Delta_{\alpha,0}$: the dynamics goes up while not at the
top of the tower, and then comes back surjectively to the basis. The
distance on $X$ is defined by $d(x,y)=\rho^{s(x,y)}$ where $\rho<1$
is fixed and $s(x,y)$, the separation time, is the number of returns
to the basis before the iterates of the points $x$ and $y$ are not in
the same element of the partition. Finally, we require a technical
distortion condition: Denoting by $g(x)$ the inverse of the jacobian
of $T$ for the measure $\mu$, we assume that $\abs{\log g(x) -\log
g(y)} \leq C d(Tx, Ty)$ for all $x,y$ in the same partition element.

With the distance $d$, the map $T$ is an isometry while going up the
tower, and expands by a factor $\rho^{-1}>1$ when going back to the
basis: it is non-uniformly expanding, the time to wait before seeing
the expansion being large on points in $\Delta_{\alpha,0}$ with
$h_\alpha$ large. In particular, denoting by $\phi(x)$ the return
time to the basis, the quantities
  \begin{equation*}
  \tail_n = \mu\{x\in \Delta_0 \st \phi(x) \geq n\} = \mu\left(\bigcup_{h_\alpha\geq n} \Delta_{\alpha,0}\right),
  \end{equation*}
called the \emph{tails of the return time}, dictate the statistical
properties of the transformation $T$. By Kac's Formula, $\tail_n$ is
summable since $\mu$ is finite by assumption. Various kinds of
behaviour can happen for $\tail_n$. For instance, in the case of
intermittent maps of parameter $\gamma\in (0,1)$, one has $\tail_n
\sim C/n^{1/\gamma}$. In general, if $\tail_n=O(n^{-q})$ for some
$q>1$, then Lipschitz functions mix at a speed $O(n^{-(q-1)})$
by~\cite{lsyoung_recurrence}, and this speed is optimal,
see~\cite{sarig_decay} and~\cite{gouezel_decay}. If $q>2$, then
$n^{-1/2}S_nf$ converges in distribution to a Gaussian, and the
variance is nonzero provided $f$ is not a coboundary. (More
generally, for convergence to a Gaussian it suffices that the return
time function $\phi$ is square-integrable, i.e., $n\tail_n$ is
summable.) When $q\in(1,2]$, more precise information is required on
$\tail_n$, leading to the following result.
\begin{thm}[\cite{gouezel_stable}]
\label{thm:stable}
Consider a Young tower with $\tail_n \sim C n^{-q}$ for some $q>1$.
There is a sequence $a_n$, and a nonempty set $\mathcal{U}$ in the
space of Lipschitz functions $f:X\to\R$ with mean zero, such that the
following holds. For each $f\in\mathcal{U}$, there exists a
nondegenerate law $Z$ such that $a_n^{-1}S_nf\to_d Z$.  Moreover,
$a_n$ and $Z$ are given as follows:

\begin{tabular}{lll}
$q>2$: & $a_n=n^{1/2}$, & $Z$ is Gaussian. \\
$q=2$: & $a_n=(n\log n)^{1/2}$, & $Z$ is Gaussian. \\
$q\in(1,2)$: & $a_n=n^{1/q}$, & $Z$ is a stable law of index $q$.
\end{tabular}
\end{thm}
The set $\mathcal{U}$ is rather big: it contains for instance all the
functions that converge to a nonzero constant along points whose
height in the tower tends to infinity.


Lower bounds for the growth of moments are well-known
(see~\cite{melbourne_nicol_large_deviations}) and can be summarized
in the following proposition. We write $\norm{u}$ for the
Lipschitz norm of a function $u$, given by
  \begin{equation*}
  \norm{u} = \sup_{x} \abs{u(x)} + \sup_{x, y} \frac{\abs{u(x)-u(y)}}{d(x,y)},
  \end{equation*}
where the supremum in the second term is restricted to those $x$ and
$y$ that belong to the same partition element. Note that, changing
the parameter $\rho$ in the definition of the distance, H\"older
functions for the old distance become Lipschitz functions for the new
one. Hence, all results that are stated in this paper for Lipschitz
functions also apply to H\"older functions.
\begin{prop}
\label{prop:lower}
Consider a Young tower with $\tail_n \sim C n^{-q}$ for some $q>1$.
Then, for all $p\in [1,\infty)$, there exists $c>0$ such that for all
$n\geq 1$
  \begin{equation*}
  \sup_{\norm{f}\leq 1, \int f\dd\mu=0} \int \abs{S_n f}^p \dd\mu\geq
  \begin{cases}
    c\max(n^{p/2}, n^{p-q+1}) & \text{ if }q>2,
    \\
    c\max( (n\log n)^{p/2}, n^{p-q+1}) & \text{ if }q=2,
    \\
    c\max( n^{p/q}, n^{p-q+1}) & \text{ if }q<2.
  \end{cases}
  \end{equation*}
\end{prop}
The phase transition in these lower bounds happens at $p^*=2q-2$ for
$q\geq 2$, and at $p^*=q$ for $q\leq 2$. Before this threshold, the
first lower bound (that corresponds to an average behavior over the
whole space) is more important, while the second one (that
corresponds to the Birkhoff sum being large on a small part of the
space) is dominating afterwards.
\begin{proof}
For the lower bound $n^{p-q+1}$, we take $f$ that is equal to $1$ on
$\bigcup_{h_\alpha \geq n} \bigcup_{i<h_\alpha} \Delta_{\alpha, i}$,
and equal to another constant on the complement of this set, to make
sure that $\int f\dd\mu=0$. Then $S_n f = n$ on $\bigcup_{h_\alpha
\geq 2n} \bigcup_{i<h_\alpha/2} \Delta_{\alpha, i}$, hence
  \begin{equation*}
  \int \abs{S_n f}^p\dd\mu \geq n^p \mu\left(\bigcup_{h_\alpha \geq 2n}
  \bigcup_{i<h_\alpha/2} \Delta_{\alpha, i}\right)
  = n^p \sum_{h\geq 2n} \frac{h}{2} \mu(\phi=h).
  \end{equation*}
Using a discrete integration by parts and the assumption $\mu(\phi
\geq n) \sim C n^{-q}$, one checks that this is $\geq c n^{p-q+1}$.

For the other bound, we fix a mean zero Lipschitz function $f$ in the
set $\mathcal{U}$ constructed in Theorem~\ref{thm:stable}. This
theorem shows the existence of $a_n$ and $Z$ nondegenerate such that
$a_n^{-1}S_nf\to_d Z$. Hence $a_n^{-p}\int \abs{S_n f}^p\dd\mu$ is
bounded from below and we get the lower bound $c a_n^p$ in all three
cases.
\end{proof}

In the case $q<2$ and $p=q$, the lower bound in the proposition is
$\int \abs{S_n f}^q \dd\mu \geq cn$. It is not sharp: for
$f\in\mathcal{U}$, $S_n f/n^{1/q}$ converges to a stable law $Z$ of
index $q$, whose $q$-th moment is infinite, hence $\int \abs{S_n
f/n^{1/q}}^q \dd\mu$ tends to infinity. To get a better lower bound,
one should study the speed of convergence of $S_n f/n^{1/q}$ to $Z$.
We can do this under stronger assumptions on the tails (this is not
surprising since it is well known that the speed of convergence to
stable laws is related to regularity assumptions on the tails of the
random variables):
\begin{prop}
\label{prop:lower_stable}
Consider a Young tower with $\tail_n = C n^{-q}+O(n^{-q-\epsilon})$
for some $q\in (1,2)$ and some $\epsilon>0$. Then there exists $c>0$
such that for all $n> 0$
  \begin{equation*}
  \sup_{\norm{f}\leq 1, \int f\dd\mu=0} \int \abs{S_n f}^q \dd\mu \geq cn \log n.
  \end{equation*}
\end{prop}
This lower bound is considerably more complicated to establish than
the ones in Proposition~\ref{prop:lower}. Since the arguments are
rather different from the rest of the paper (essentially, they reduce
to a proof of a Berry-Esseen like bound for $S_n f/n^{1/q}$), we
defer the proof of the proposition to
Appendix~\ref{sec:lower_stable}. The assumptions of this proposition
are for instance satisfied for the classical Pomeau-Manneville
intermittent maps~\cite{liverani_saussol_vaienti,pomeau_manneville}.
(See for example~\cite[Proposition~11.12]{melbourne_terhesiu}.)

\medskip

For $q=2$, the bound $c\max( (n\log n)^{p/2}, n^{p-q+1})$ is known to
be optimal for all $p$, see Remarks~\ref{rmk:BCD}
and~\ref{rmk:previous} below. Also, for $q>2$, the bound
$c\max(n^{p/2}, n^{p-q+1})$ is known to be optimal for all $p\neq
2q-2$.  The remaining cases are much more subtle, and are solved for
the first time in this paper.
We note that for $q>2$ and $p=2q-2$,~\cite{gouezel_chazottes_concentration}
obtains an additional upper bound for the weak moment of $S_n f$,
which implies for $p>2q-2$ the upper bound $C n^{p-q+1}$, in
accordance with the lower bound. Moreover, the very precise methods
of~\cite{gouezel_chazottes_concentration} seemed to indicate that the
upper bound for the weak moment at $p=2q-2$ was optimal, and that the
discrepancy with the lower bound was due to a suboptimality of the
(naive) lower bound. We prove below that this is not the case.
\begin{thm}
\label{main_thm:moments} Consider a Young tower with $\tail_n = O(n^{-q})$ for some
$q>1$. Then, for all $p\in [1,\infty)$, there exists $C>0$ such that
for any Lipschitz function $f$ with $\norm{f}\leq 1$ and $\int
f\dd\mu=0$, for all $n\geq 0$,
  \begin{equation*}
  \int \abs{S_n f}^p \dd\mu\leq
  \begin{cases}
    C\max(n^{p/2}, n^{p-q+1}) & \text{ if }q>2,
    \\
    C\max( (n\log n)^{p/2}, n^{p-q+1}) & \text{ if }q=2,
    \\
    C\max( n^{p/q}, n^{p-q+1}) & \text{ if }q<2\text{ and }p\neq q.
  \end{cases}
  \end{equation*}
If $q<2$, we have for all $t>0$
  \begin{equation}
  \label{eq:weak_q}
  \mu\{ x \st \abs{S_n f(x)}\geq t\} \leq C t^{-q}n
  \end{equation}
and therefore
  \begin{equation*}
  \int \abs{S_n f}^q\dd\mu \leq C n\log n.
  \end{equation*}
\end{thm}
Our upper bounds all match the corresponding lower bounds given in
Propositions~\ref{prop:lower} and~\ref{prop:lower_stable}, and are
therefore optimal.

Note that, in the proofs, if is sufficient to understand what happens
at the critical exponent $p^*=2q-2$ for $q> 2$: a control on the
$L^{2q-2}$-norm for $q>2$ readily implies the control for any $p\in
[1,\infty)$ thanks to the trivial inequalities
  \begin{equation}
  \label{eq:deduce_moments_trivial}
  \norm{u}^p_{L^p} \leq \norm{u}^{p-r}_{L^\infty}\norm{u}^r_{L^r} \text{ if $p>r$},\quad
  \norm{u}_{L^p} \leq \norm{u}_{L^r} \text{ if $p<r$}.
  \end{equation}
In the same way, for $q<2$, the control~\eqref{eq:weak_q} on the weak
$q$-th moment implies the corresponding $L^p$ controls for any $p\in
[1,\infty)$ thanks to the equality
  \begin{equation}
  \label{eq:deduce_moments}
  \int \abs{u}^p \dd\mu = p\int_{s=0}^{\norm{u}_{L^\infty}}
  s^{p-1}\mu\{\abs{u}>s\}\dd s.
  \end{equation}
This formula would also apply in the $q>2$ case (combined with the
control of $\mu\{\abs{u}>s\}$ coming from the estimate at the
exponent $p^*$ and the Markov inequality), but it gives worse
constants than~\eqref{eq:deduce_moments_trivial} in this case.

On the other hand, for $q=2$, the bound $\sqrt{n \log n}$ for the
second moment does not give the desired upper bound for $p>2$ (using
the formulas~\eqref{eq:deduce_moments_trivial}
or~\eqref{eq:deduce_moments}, one only gets the upper bound $\int
\abs{S_n f}^p \dd\mu \leq Cn^{p-1} \log n$, with an extra $\log n$).

\medskip

As an immediate consequence of the bounds on moments at the critical exponent,
we obtain convergence of moments for all lower exponents.

\begin{cor} \label{cor:convergence}
Consider a Young tower with $\tail_n \sim C n^{-q}$ for some $q>1$.
Suppose that $f$, $a_n$ and $Z$ are as in Theorem~\ref{thm:stable}.
Then $\int \abs{a_n^{-1}S_nf}^p \dd\mu\to \E(\abs{Z}^p)$ for all
$p<p^*$ where $p^*=2q-2$ for $q\ge 2$ and $p^*=q$ for $q\in(1,2)$.

In particular, there exist nonzero constants $C=C_{p,q}$ such that
\begin{itemize}
\item  if $q>2$, then $\int \abs{S_nf}^p \dd\mu\sim C n^{p/2}$
    for all $p<2q-2$.
\item  if $q=2$, then $\int \abs{S_nf}^p \dd\mu\sim C (n\log
    n)^{p/2}$ for all $p<2$.
\item  if $q\in(1,2)$, then $\int\abs{S_nf} ^p\dd\mu\sim C
    n^{p/q}$ for all $p<q$.
\end{itemize}
\end{cor}

\begin{proof}
As in~\cite{MTorok12}, this is an immediate consequence of
Theorem~\ref{thm:stable}, together with the fact that
$\int\abs{a_n^{-1/2}S_nf}^{p'}\dd\mu$ is bounded for any $p'\in (p,
p^*)$ as guaranteed by Theorem~\ref{main_thm:moments} (for $q\geq 2$,
one can even take $p'=p^*$).
\end{proof}

\begin{rmk} \label{rmk:BCD}
Previously, no results were available on convergence of moments for
$q<2$. The case $q>2$ in Corollary~\ref{cor:convergence} recovers a
result of~\cite{MTorok12} and the result for $q=2$ was obtained
by~\cite{Chernov_preprint} in the context of dispersing billiards
with cusps. \cite{Chernov_preprint} consider also the critical
exponent $p=2$ for dispersing billiards with cusps, and prove for
this example that the limiting second moment is twice the moment of
the limiting Gaussian: $\int\abs{(n\log n)^{-1/2}S_nf}^2 \dd\mu\to
2\E(\abs{Z}^2)$. This particular behaviour is due to the very
specific geometric structure of the billiard.
\end{rmk}

\begin{rmk} \label{rmk:previous}
Certain aspects of Theorem~\ref{main_thm:moments} and
Corollary~\ref{cor:convergence} do not require the full strength of
the assumption that there is an underlying Young tower structure. We
can consider the more general situation where $f$ is a mean zero
observable lying in $L^\infty$ such that $\abs*{\int f\,g\circ T^n
\dd\mu}\le C\norm{g}_{L^\infty}n^{-(q-1)}$ for all $g\in L^\infty$,
$n\ge1$. (Such a condition is satisfied for $f$ Lipschitz when $X$ is
a Young tower with $\tail_n =O(n^{-q})$.)

In the case $q=2$, this weaker condition is sufficient to recover all
the moment estimates (and hence the convergence of moments for $p<2$)
described above. By~\cite[Lemma~2.1]{M09b}, $\int
\abs{S_nf}^p\dd\mu\ll n^{p-1}$ for $p>2$ and $\int
\abs{S_nf}^2\dd\mu\ll n\log n$.

In the case $q>2$, $\int \abs{S_nf}^p\dd\mu \ll n^{p-q+1}$ for
$p>2q-2$ by~\cite[Lemma~2.1]{M09b}, and $\int \abs{S_nf}^p\dd\mu\le
Cn^{p/2}$ for $p<2q-2$ by~\cite{MTorok12}, Again it follows that all
moments converge for $p<2q-2$.

After we completed this article, we learned that, using techniques
that are completely different from the ones we develop, Dedecker and
Merlev\`ede~\cite{dedecker_merlevede_moment} also obtain the controls
on moments given in Theorem~\ref{main_thm:moments}, essentially under
an assumption of the form $\abs*{\int f\,g\circ T^n\dd\mu}\le
C\norm{g}_{L^\infty}n^{-(q-1)}$. Their arguments (initially developed
to control the behavior of the empirical measure) rely on general
probabilistic inequalities for sums of random variables, and can
apparently not give the concentration inequalities of
Theorem~\ref{main_thm:concentration} below.
\end{rmk}

\begin{rmk}
Proposition~\ref{prop:lower} and Theorem~\ref{main_thm:moments}
clarify certain results in the Physics literature. As
in~\cite{MTorok12}, our results go over to flows, and apply in
particular to infinite horizon planar periodic Lorentz gases.   These
can be viewed as suspension flows over Young towers with $\tail_n\sim
Cn^{-2}$ so we are in the case $q=2$.  In particular, if $r(t)$
denotes position at time $t$, then $(t\log t)^{-1/2}r(t)\to_d Z$
where $Z$ is a nondegenerate Gaussian~\cite{szasz_varju_infinite}.
\cite{Armstead_etal03} consider growth rate of moments for $r(t)$,
but neglecting logarithmic factors.  Defining
$\gamma_p=\lim_{t\to\infty}\log \int\abs{r(t)}^p \dd\mu/\log t$, they
argue heuristically that $\gamma_p=\max\{p/2,p-1\}$ in accordance
with our main results. \cite{Courbage_etal08} conducted numerical
simulations to verify the growth rates of the moments, including
logarithmic factors, but based on the belief that $\int\abs{r(t)}^p
\dd\mu$ scales like $(n\log n)^{p/2}$ for all $p$, whereas we have
shown that this is correct only for $p\le2$.

Two other examples of billiards that are modelled by Young towers
with $\tail_n\sim Cn^{-2}$ are Bunimovich stadia (discrete and
continuous time)~\cite{balint_gouezel} and billiards with cusps
(discrete time)~\cite{BalintChernovDolgopyat11,Chernov_preprint}.
Again, our results apply to these situations with $q=2$.
\end{rmk}

\medskip

The above optimal upper bounds for moments, dealing with Birkhoff
sums, can be extended to concentration estimates, for any (possibly
non-linear) function of the point and its iterates. More precisely,
consider a function $K(x_0,x_1,\dotsc)$ (depending on finitely or
infinitely many coordinates) which is separately Lipschitz: for all
$i$, there exists a constant $\Lip_i(K)$ such that, for all
$x_0,x_1,\dotsc$ and $x'_i$,
  \begin{equation*}
  \abs{K(x_0,x_1,\dotsc,x_{i-1}, x_i, x_{i+1},\dotsc)
  -K(x_0,x_1,\dotsc,x_{i-1}, x'_i, x_{i+1},\dotsc)}
  \leq \Lip_i(K) d(x_i, x'_i).
  \end{equation*}
If $K$ does not depend on some variable $x_i$, we set by convention
$\Lip_i(K)=0$.

The function $K$ is defined on the space $\tilde X=X^{\N}$. This
space carries a natural probability measure, describing the
deterministic dynamics once the starting point is chosen at random
according to $\mu$, i.e., $\tilde\mu\coloneqq \mu\otimes
\delta_{Tx}\otimes \delta_{T^2 x}\otimes \dotsm$. Let
  \begin{equation*}
  \E(K)=\int_X K(x,Tx,\dotsc)\dd\mu(x)=\int_{\tilde X} K \dd\tilde\mu.
  \end{equation*}
This is the average of $K$ with respect to the natural measure of the
system. We are interested in the deviation of $K(x, Tx,\dotsc)$ from
its average $\E(K)$. For instance, if $K(x_0,\dotsc, x_{n-1})=\sum
f(x_i)$, then $K(x, Tx,\dotsc)$ is simply the Birkhoff sum $S_n f$.
It is separately Lipschitz if $f$ is Lipschitz, with Lipschitz
constants $\Lip_i(K)=\Lip(f)$ for $0\leq i\leq n-1$, and
$\Lip_i(K)=0$ otherwise.

\begin{thm}
\label{main_thm:concentration}
Consider a Young tower with $\tail_n =O(n^{-q})$ for some $q>1$.
Then, for all $p\in [1,\infty)$, there exists $C>0$ such that, for
all separately Lipschitz function $K$,
\begin{itemize}
\item if $q>2$,
  \begin{equation*}
  \int \abs{K(x,Tx,\dotsc)-\E K}^{p}\dd\mu \leq \begin{cases}
    C\left(\sum \Lip_i(K)^2\right)^{p/2}&\text{ if }p\leq 2q-2,\\
    C\left(\sum \Lip_i(K)^2\right)^{q-1} \left(\sum \Lip_i(K)\right)^{p-(2q-2)}&\text{ if }p\geq 2q-2.
  \end{cases}
  \end{equation*}
\item if $q=2$, the quantity $\int \abs{K(x,Tx,\dotsc)-\E
    K}^{p}\dd\mu$ is bounded by
    \begin{equation*}
    \begin{cases}
    C\left(\sum \Lip_i(K)^2\right)^{p/2}\left[1+\log\left(\sum \Lip_i(K)\right)
      -\log\left(\sum \Lip_i(K)^2\right)^{1/2}\right]^{p/2}&\text{ if }p\leq 2,\\
    C\left(\sum \Lip_i(K)^2\right) \left(\sum \Lip_i(K)\right)^{p-2}&\text{ if }p>2.
  \end{cases}
  \end{equation*}
\item if $q<2$, then for all $t>0$
  \begin{equation}
  \label{eq:concent_q<2_weak}
  \mu\{x \st \abs{K(x,Tx,\dotsc)-\E K} \geq t\} \leq C t^{-q} \sum_i \Lip_i(K)^q
  \end{equation}
and therefore $\int \abs{K(x,Tx,\dotsc)-\E K}^{p}\dd\mu$ is
bounded by
  \begin{equation*}
  \begin{cases}
    C\left(\sum \Lip_i(K)^q\right)^{p/q}&\text{ if }p<q,\\
    C\left(\sum \Lip_i(K)^q\right)\left[1+\log\left(\sum \Lip_i(K)\right)
      -\log\left(\sum \Lip_i(K)^q\right)^{1/q}\right]&\text{ if }p=q,\\
    C\left(\sum \Lip_i(K)^q\right) \left(\sum \Lip_i(K)\right)^{p-q}&\text{ if }p>q.
  \end{cases}
  \end{equation*}
\end{itemize}
\end{thm}
Note that $\abs{K-\E(K)}$ is trivially bounded by $\sum_i \Lip_i(K)$.
Hence, when $q>2$, it is sufficient to prove the estimates for
$p=2q-2$, as the other ones follow
using~\eqref{eq:deduce_moments_trivial}. In the same way, for $q<2$,
it suffices to prove the weak moment
bound~\eqref{eq:concent_q<2_weak}, thanks
to~\eqref{eq:deduce_moments}. On the other hand, for $q=2$, the
inequality for $p=2$ is not sufficient to obtain the result for
$p>2$.

There are logarithmic terms in some of the above bounds when $q\leq
2$. This is not surprising, since such terms are already present in
the simpler situation of Birkhoff sums, in
Theorem~\ref{main_thm:moments}. The precise form of these logarithmic
terms may seem surprising at first sight, but it is in fact natural
since such a bound has to be homogeneous: The logarithmic term should
be invariant if one replaces $K$ with $\lambda K$, and therefore each
$\Lip_i(K)$ with $\lambda \Lip_i(K)$. This would not be the case for
the simpler bound $\log(\sum \Lip_i(K))$. When $\Lip_i(K)$ does not
depend on $i$, the bound $\log\left(\sum \Lip_i(K)\right)
-\log\left(\sum \Lip_i(K)^q\right)^{1/q}$ reduces to $(1-1/q) \log
n$, a constant multiple of $\log n$ as we may expect.

Compared to moment controls, concentration results for arbitrary
functions $K$ have a lot more applications, especially when $K$ is
non-linear. We refer the reader
to~\cite[Section~7]{gouezel_chazottes_concentration} for a
description of such applications.

Theorem~\ref{main_thm:concentration} implies
Theorem~\ref{main_thm:moments} (just take $K(x_0,\dotsc,
x_{n-1})=\sum f(x_i)$). However, the proof of
Theorem~\ref{main_thm:moments} is considerably simpler, and motivates
some techniques used in the proof of
Theorem~\ref{main_thm:concentration}. Hence, we prove both
theorems separately below. While some cases of
Theorem~\ref{main_thm:moments} are already known (especially the case
$q=2$, see Remark~\ref{rmk:previous}), we nevertheless give
again a full proof of these cases, for completeness and with the
concentration case in mind.

\medskip

The proofs of our results rely on two main tools: a dynamical one
(very precise asymptotics of renewal sequences of operators) and a
probabilistic one (inequalities for martingales, of
Burkholder-Rosenthal and von Bahr-Esseen type). In addition, for the
concentration inequalities, we require analytic tools such as
maximal inequalities and interpolation results, since the Lipschitz
constants $\Lip_a(K)$ may vary considerably with $a$, which makes more usual
inequalities too crude. All these tools are presented in
Section~\ref{sec:preliminaries}. Theorem~\ref{main_thm:moments} is
proved in Section~\ref{sec:moment}, and
Theorem~\ref{main_thm:concentration} is proved in
Section~\ref{sec:concentration}.

\section{Preliminaries}
\label{sec:preliminaries}

\subsection{Renewal sequences of operators}
\label{subsec:renewal}

In this paragraph, we summarize the results on renewal sequences of
operators that we need later on. They are proved
in~\cite{sarig_decay, gouezel_decay, gouezel_these}.

Consider a Young tower $T:X\to X$. The associated transfer operator
$\boL$, adjoint to the composition by $T$, is given by
  \begin{equation*}
  \boL u(x) = \sum_{Ty=x} g(y) u(y).
  \end{equation*}
Denoting by $g^{(n)}(x)=g(x)\dotsm g(T^{n-1}x)$ the inverse of the
jacobian of $T^n$, one has $\boL^n u(x) = \sum_{T^n y=x} g^{(n)}(y)
u(y)$. Iterating the inequality $\abs{\log g(x)-\log g(y)}\leq C
d(Tx, Ty)$ and using the uniform expansion when a trajectory returns
to the basis, one has the following bounded distortion property:
there exists $C>0$ such that, for all $n$, for all points $x$ and $y$
in the same cylinder of length $n$ (i.e., for $i<n$, the points $T^i
x$ and $T^i y$ are in the same partition element),
  \begin{equation*}
  \abs{\log g^{(n)}(x) - \log g^{(n)}(y)} \leq C d(T^n x, T^n y).
  \end{equation*}

Among the trajectories of $T$, the only non-trivial behavior is
related to the successive returns to the basis. Define a first return
transfer operator at time $n$ by $R_n u(x)= \sum_{T^n y=x} g^{(n)}(y)
u(y)$ where $x\in  \Delta_0$ and the sum is over those preimages $y$
of $x$ that belong to $\Delta_0$ but $T^i y \not\in \Delta_0$ for
$1\leq i\leq n-1$. Since $R_n$ only involves preimages $y$ with
$\phi(y)=n$, its operator norm $\norm{R_n}$ with respect to the
Lipschitz norm satisfies $\norm{R_n}\leq C \mu(\phi=n)$. In
particular, $R_n$ is easy to understand.

Define a partial transfer operator $T_n = 1_{\Delta_0}\boL^n
1_{\Delta_0}$. It can be written as $T_n u(x)=\sum_{T^n y= x}
g^{(n)}(y)u(y)$, where $x$ and $y$ all have to belong to $\Delta_0$.
Decomposing a trajectory from $\Delta_0$ to $\Delta_0$ into
successive excursions, one gets
  \begin{equation*}
  T_n = \sum_{k=1}^n \sum_{\ell_1+\dotsb+\ell_k=n} R_{\ell_1}\dotsm R_{\ell_k}.
  \end{equation*}
Formally, this is equivalent to the equality $\sum T_n z^n = (I-\sum
R_k z^k)^{-1}$. This makes it possible to understand $T_n$. Denote by
$\Pi$ the projection on constant functions on $\Delta_0$, given by
$\Pi u(x) = \int_{\Delta_0} u \dd \mu / \mu(\Delta_0)$.

The following proposition is~\cite[Proposition 2.2.19 and Remark
2.4.8]{gouezel_these} in the specific case of polynomial growth rate
(this proposition also holds for more exotic asymptotics such as
$O(n^{-q} \log n)$ -- it follows that most results of our paper could
be extended to such speeds).
\begin{prop}
\label{prop:conseq_banach}
Assume that $\mu(\phi \geq n)=O(n^{-q})$ for some $q>1$. Then
$\norm{T_{n+1}-T_n}=O(n^{-q})$ and $\norm{T_n -\Pi
T_n\Pi}=O(n^{-q})$.
\end{prop}
In particular, $\norm{T_{n+1}-T_n}$ is summable, hence $T_n$
converges. Its limit is $\mu(\Delta_0) \Pi$.

Consider now a general function $u$ and a point $x\in \Delta_0$, we
wish to describe $\boL^n u(x) = \sum_{T^n y=x} g^{(n)}(y) u(y)$.
Splitting the trajectory of $y$ into a first part until the first
entrance in $\Delta_0$, of length $b\geq 0$, and then a second part
starting from $\Delta_0$ at time $b$ and coming back to $\Delta_0$ at
time $n$, we obtain a decomposition
  \begin{equation}
  \label{eq:decompose_boLn}
  1_{\Delta_0} \boL^n = \sum_{\ell+b=n} T_\ell B_b.
  \end{equation}
The operator $B_b$ is given by $B_b u(x) = \sum_{T^b
y=x}g^{(b)}(y) u(y)$, the sum being restricted to those preimages
whose first entrance in $\Delta_0$ is at time $b$ (the projection in
the basis of those points necessarily has $\phi>b$). By bounded
distortion, one gets
  \begin{equation}
  \label{eq:controle_Bb}
  \norm{B_b} \leq C \mu(\phi > b).
  \end{equation}

\subsection{Weak \texorpdfstring{$L^p$}{Lp} spaces}

If a function $u$ belongs to $L^p$ on a probability space, then
$\Pbb(\abs{u}>s) \leq s^{-p}\E(\abs{u}^p)$ by Markov's inequality. On
the other hand, this condition $\Pbb(\abs{u}>s) = O(s^{-p})$ is not
sufficient to belong to $L^p$. For instance, a stable law of index
$p\in (1,2)$ satisfies $\Pbb(\abs{Z}>s) \sim c s^{-p}$, it readily
follows that it does not belong to $L^p$.

We say that a random variable $u$ belongs to weak $L^p$ if
$\Pbb(\abs{u}>s) = O(s^{-p})$. We write
  \begin{equation*}
  \norm{u}^p_{L^{p,w}} = \sup_s s^p \Pbb(\abs{u}>s).
  \end{equation*}
This is the analogue of the $L^p$ norm in this context. It satisfies
$\norm{u}_{L^{p,w}}\leq \norm{u}_{L^p}$. In general,
$\norm{\cdot}_{L^{p,w}}$ is not a norm (i.e., it does not satisfy the
triangular inequality), however it is equivalent to a norm when $p>1$
(see for instance~\cite[Paragraph V.3]{stein_weiss_fourier}). The
weak $L^p$ space is a particular instance of \emph{Lorentz spaces},
corresponding to the space $L^{p,\infty}$ in the standard notation.

Apart from its natural appearance when considering stable laws, a
major role of the weak $L^p$ space comes from interpolation theory.
The following is a particular case of the Marcinkiewicz interpolation
theorem, see for instance~\cite[Theorem V.2.4]{stein_weiss_fourier}.
\begin{thm}
\label{thm:interpolation}
If a linear operator is bounded from $L^1(\mu)$ to $L^{1,w}(\nu)$ and
from $L^\infty(\mu)$ to $L^\infty(\nu)$, then it is bounded from
$L^p(\mu)$ to $L^p(\nu)$ for any $1<p<\infty$.
\end{thm}

This result can for instance be used to prove the boundedness of the
Hardy-Littlewood maximal function on any $L^p$ space, $1<p\leq
\infty$, since boundedness from $L^1$ to $L^{1,w}$ and from
$L^\infty$ to itself hold. We recall the statement in the case of
$\Z$, since we need it later on. See for instance~\cite[Theorem
II.3.7]{stein_weiss_fourier}.
\begin{thm}
\label{thm:maximal}
To a sequence $(u_n)_{n\in \Z}$, associate the sequence
  \begin{equation*}
  Mu(n) = \sup_{h\geq 0} \frac{1}{2h+1}\sum_{i=n-h}^{n+h} \abs{u_i}.
  \end{equation*}
For all $p\in (1,+\infty]$, there exists a constant $C$ such that
$\norm{M u}_{\ell^p} \leq C \norm{u}_{\ell^p}$ for any sequence $u\in
\ell^p$.
\end{thm}

\subsection{Martingale inequalities}

Given a decreasing sequence of $\sigma$-algebras $\boF_0 \supset
\boF_1 \supset \dotso$ on a probability space, a sequence of reverse
martingale differences with respect to this filtration is a sequence
of random variables $D_k$ such that $\E(D_k \given \boF_{k+1}) = 0$.
This is a kind of one-sided independence condition. Moment
inequalities, similar to classical inequalities for independent
random variables, hold in this setting.

We will use the following Burkholder-Rosenthal inequality:
\begin{thm}
\label{thm:BR}
For any $Q\geq 2$, there exists a constant $C$ such that any sequence
of reverse martingale differences satisfies
  \begin{equation*}
  \E\left|\sum D_k\right|^Q \leq C \E\left(\sum \E(D_k^2 \given \boF_{k+1})\right)^{Q/2}
  + C \E(\max \abs{D_k}^Q ).
  \end{equation*}
As a consequence,
  \begin{equation}
  \label{eq:burkh_useful}
  \E\left|\sum D_k\right|^Q \leq C \left(\sum \norm{\E(D_k^2 \given \boF_{k+1})}_{L^\infty}\right)^{Q/2}
  + C \sum \norm{\E(\abs{D_k}^Q \given \boF_{k+1})}_{L^\infty}.
  \end{equation}
\end{thm}
The first statement is due to
Burkholder~\cite[Theorem~21.1]{burkholder}. The second (much weaker)
statement readily follows, and is sufficient for our purposes. One
interest of the second formulation is that the two terms look the
same: in the applications we have in mind, we will control
simultaneously $\norm{\E(D_k^2 \given \boF_{k+1})}_{L^\infty}$ and
$\norm{\E(\abs{D_k}^Q \given \boF_{k+1})}_{L^\infty}$.

For $Q\in (1,2)$, the (easier) analogue of the above theorem is the
inequality of von Bahr and Esseen~\cite{vonbahr_esseen} stating that
  \begin{equation*}
  \E\abs*{\sum D_k}^Q \leq C \sum \E\abs{D_k}^Q.
  \end{equation*}
However, we will rather need a version of this inequality involving
weak $L^Q$ norms (since the main part of
Theorem~\ref{main_thm:moments} in the case $q<2$ is the
inequality~\eqref{eq:weak_q}, controlling the weak $L^q$ norm of $S_n
f$). Such an inequality holds:
\begin{thm}
\label{prop:weak_vBE}
For any $Q\in (1,2)$, there exists a constant $C$ such that any
sequence of reverse martingale differences $D_k$ satisfies
  \begin{equation}
  \label{eq:vBE}
  \norm*{\sum D_k}^Q_{L^{Q,w}} \leq C \sum_k \norm{D_k}^Q_{L^{Q,w}}.
  \end{equation}
\end{thm}
\begin{proof}
This is a consequence of existing results in the literature, as we
now explain. First, the $L^{Q,w}$-seminorm is not a norm, which can
be a problem for the proof of inequalities involving an arbitrary
number of terms. However, it is equivalent to a true norm, the
Lorentz norm $L^{Q,\infty}$ (see~\cite[Paragraph
V.3]{stein_weiss_fourier})), so this is not an issue.

\cite[Theorem~7(1) on Page 39]{braverman_rearrangement} proves that
the space $L^{Q,w}=L^{Q,\infty}$ satisfies the von Bahr-Esseen
property of index $Q$, i.e., the inequality~\eqref{eq:vBE} holds
whenever the $D_k$ are independent centered random variables.

Consider now a sequence of reverse martingale differences $D_k$. Let
$(\tilde D_k)$ be independent random variables, such that $\tilde
D_k$ is distributed as $D_k$. \cite[Theorem
6.1]{astashkin_weak_lp_vonbahr} shows that
  \begin{equation*}
  \norm*{\sum D_k}_{L^{Q,w}} \leq C \norm*{\sum \tilde D_k}_{L^{Q,w}}.
  \end{equation*}
As the random variables $\tilde D_k$ are independent, they
satisfy~\eqref{eq:vBE} by~\cite{braverman_rearrangement}. The same
inequality follows for $D_k$.
\end{proof}

\subsection{Miscellaneous}

We use repeatedly the following classical lemma, which is
readily proved by a discrete integration by parts.
\begin{lem}
\label{lem:karamata}
Let $c_h$ be a sequence of nonnegative real numbers with $\sum_{h>n}
c_h = O(n^{-q})$ for some $q>1$. Then, for all $\alpha<q$, one has
$\sum_{h>n} h^\alpha c_h = O(n^{\alpha-q})$. Moreover, for all
$\alpha>q$, one has $\sum_{h<n} h^\alpha c_h = O(n^{\alpha-q})$.
Finally, $\sum_{h<n} h^q c_h=O(\log n)$.
\end{lem}

We also use the following fact: If $c_n$ is a summable sequence of
nonnegative real numbers and $p\geq1$,
  \begin{equation}
  \label{eq:stable_power}
  \abs*{\sum c_n u_n}^p \leq \left(\sum c_n\right)^{p-1} \sum c_n \abs{u_n}^p.
  \end{equation}
Indeed, this follows from the convexity of $x\mapsto x^p$ for $\sum
c_n=1$, and the general case follows.

\section{Moment bounds}
\label{sec:moment}

Our goal in this section is to prove Theorem~\ref{main_thm:moments}.
We therefore fix a Young tower with $\tail_n=O(n^{-q})$ for some
$q>1$.

The convolution of two sequences $(c_n)_{n\geq 0}$ and $(d_n)_{n \geq
0}$ is the sequence $c\star d$ given by
  \begin{equation*}
  (c\star d)_n = \sum_{i=0}^n c_i d_{n-i}.
  \end{equation*}

We write $c_n^{(q)}$ for a generic sequence of the form
$C/(n+1)^{q}$, for a generic $C$ that can change from one occurrence
to the next, even on the same line, but only finitely many times in
the whole article. We use repeatedly the fact that the
convolution of two such sequences is bounded by a sequence of the
same form. This fact reads
  \begin{equation}
  \label{seq:convole}
  (c^{(q)}\star c^{(q)})_n = \sum_{i=0}^n c^{(q)}_i c^{(q)}_{n-i} \leq c_n^{(q)}.
  \end{equation}
(Note that the sequence $c_n^{(q)}$ on the right is not the same as
the sequences on the left, in accordance with the above convention.)

We wish to understand the moments of Birkhoff sums $S_n f$. Since
martingale inequalities are very powerful, we will reduce to such
martingales in the most naive way. Let $\boF_k = T^{-k}(\boF_0)$
(where $\boF_0$ is the Borel $\sigma$-algebra), a function is
$\boF_k$-measurable if and only if it can be written as $u\circ T^k$
for some function $u$. We have
  \begin{equation*}
  S_n f = \sum_{k=0}^{n-1} (\E(S_n f \given \boF_k) - \E(S_n f \given \boF_{k+1}))
  + \E(S_n f \given \boF_n)
  = \sum_{k=0}^{n-1} A_k \circ T^k + \E(S_n f \given \boF_n),
  \end{equation*}
for some functions $A_k$ that we now describe. Note that this is
a decomposition as a sum of reverse martingale differences, hence the
moments of $S_n f$ will essentially be controlled by those of $A_k$.

Let $\boL$ be the transfer operator, it satisfies $\E(u \given
\boF_1) = (\boL u) \circ T$. Hence, for $k<n$,
  \begin{equation*}
  \E(S_n f \given \boF_k) - \E(S_n f \given \boF_{k+1})
  = \left(\sum_{i=0}^{k} \boL^i f\right)\circ T^k
    - \left(\sum_{i=1}^{k+1} \boL^i f\right) \circ T^{k+1},
  \end{equation*}
giving
  \begin{equation}
  \label{eq:Ak}
  A_k = \sum_{i=0}^{k} \boL^i f - \left(\sum_{i=1}^{k+1} \boL^i f\right)\circ T.
  \end{equation}

Let us define a function $F_k= \sum_{i=0}^{k} \boL^i f$, this is the
main function to understand.
\begin{lem}
\label{lem_Fk}
If $x$ is at height $h$ and $Tx \in \Delta_0$, then
  \begin{equation*}
  F_k(x) - F_{k+1}(Tx) = O(1+ h\wedge k).
  \end{equation*}
\end{lem}
\begin{proof}
Clearly $\norm{F_k-F_{k+1}\circ T}_{L^\infty}\le 2(1+k)\norm{f}_{L^\infty}$.
We have to prove that $F_k(x) - F_{k+1}(Tx) = O(1+h)$.

First, we estimate $1_{\Delta_0}F_k$.
We use the formalism of renewal
transfer operators introduced in Paragraph~\ref{subsec:renewal}. As
in~\eqref{eq:decompose_boLn}, we write $1_{\Delta_0} \boL^n =
\sum_{\ell+b=n} T_\ell B_b$, where $T_\ell$ counts the returns to the
basis at time $\ell$, and $B_b$ is an average over preimages at time
$b$ that did not return to the basis in between. Write $\Pi$ for the
projection on constant functions on $\Delta_0$.
Proposition~\ref{prop:conseq_banach} shows that the operator $E_\ell
= T_\ell-\Pi T_\ell \Pi$ satisfies $\norm{E_\ell} \leq c^{(q)}_\ell$.
We get
  \begin{equation*}
  1_{\Delta_0}F_k = \sum_{i=0}^{k} 1_{\Delta_0}\boL^i f
  =\sum_{\ell + b \le k} T_\ell B_b f
  =\sum_{\ell + b \le k} \Pi T_\ell \Pi B_b f + \sum_{\ell + b \le k}E_\ell B_b f.
  \end{equation*}
Since $\norm{B_b f} \leq c_b^{(q)}$ by~\eqref{eq:controle_Bb} and
$\norm{E_\ell} \leq c_\ell^{(q)}$, the second sum is uniformly
$O(1)$. For the first sum, the function $\Pi T_\ell \Pi B_b f$ is
constant by definition, and can be written as $t_\ell u_b(f)$ for
$t_\ell = \int_{\Delta_0} T_\ell 1 \dd\mu/\mu(\Delta_0)$ and $u_b(f)
= \int_{\Delta_0} B_b f \dd\mu/\mu(\Delta_0)$. We have obtained
  \begin{equation*}
  1_{\Delta_0}F_k = \sum_{\ell + b\le k} t_\ell u_b(f) + O(1),
  \end{equation*}
where $t_\ell$ is uniformly bounded, and $u_b(f)$ is summable (with
sum at most $\int \abs{f}\dd\mu/\mu(\Delta_0)$).

Consider now an arbitrary $x$, at height $h<k$, and with $Tx \in
\Delta_0$. Then $F_k(x) = F_{k-h}(\pi x) + O(h)$ where $\pi x$ is the
projection of $x$ in the basis of the tower, i.e., the unique
preimage of $x$ under $T^h$. We get
  \begin{align*}
  F_k(x) - F_{k+1}(Tx) &=
  \sum_{\ell +b \le  k-h} t_\ell u_b(f) -\sum_{\ell+b \le  k+1} t_\ell u_b(f) + O(1+h)
  \\&
  = \sum_{k-h< \ell+b \leq k+1} t_\ell u_b(f) + O(1+h).
  \end{align*}
For each $b$, there are at most $h+1$ values of $\ell$ for which
$k-h< \ell+b \leq k+1$. Since $t_\ell$ is bounded, we obtain
  \begin{equation*}
  \abs{F_k(x) - F_{k+1}(Tx)} \leq (h+1) \sum_b \abs{u_b(f)} +O(1+h) = O(1+h).
  \qedhere
  \end{equation*}
\end{proof}

\subsection{The case \texorpdfstring{$q>2$}{q>2}}

In this paragraph, we prove Theorem~\ref{main_thm:moments} in the
case $q>2$. It suffices to prove the desired estimate for $p=2q-2$,
since the other estimates follow
using~\eqref{eq:deduce_moments_trivial}.

We start from the decomposition
  \begin{equation}
  \label{eq:dec_Snf}
  S_n f=\sum A_k\circ T^k + \E(S_nf \given \boF_n).
  \end{equation}

First, we control the last term, which is easier. Write $Q=2q-2$, we
have
  \begin{equation*}
  \norm{\E(S_nf \given \boF_n)}_{L^Q}
  =\norm*{\left(\sum_{k=1}^n \boL^k f\right)\circ T^n}_{L^Q}
  =\norm*{\sum_{k=1}^n \boL^k f}_{L^Q}
  \leq \sum_{k=1}^n \norm{\boL^k f}_{L^Q}.
  \end{equation*}
One can use transfer operators techniques, or argue directly as
in~\cite{melbourne_nicol_large_deviations}: since the speed of decay
of correlations against bounded functions is $O(1/n^{q-1})$
by~\cite{lsyoung_recurrence}, we have
  \begin{equation*}
  \int \abs{\boL^k f}^Q \dd\mu \leq \norm{f}_\infty^{Q-1} \int \sgn(\boL^k f)\cdot \boL^k f \dd\mu
  \leq C/k^{q-1}.
  \end{equation*}
Hence, $\norm{\boL^k f}_{L^Q} \leq C/k^{(q-1)/Q}=k^{-1/2}$, giving
$\norm{\E(S_nf \given \boF_n)}_{L^Q} \leq C n^{1/2}$.

Then, we turn to the first sum $\sum A_k \circ T^k =\sum D_k$
in~\eqref{eq:dec_Snf}. It is a sum of reverse martingale differences,
hence we may apply Burkholder-Rosenthal inequality in the form
of~\eqref{eq:burkh_useful}:
  \begin{equation}
  \label{eq:BR}
  \E\abs*{\sum D_k}^Q \leq C \left(\sum \norm{\E(D_k^2 \given \boF_{k+1})}_{L^\infty}\right)^{Q/2}
  + C \sum \norm{\E(\abs{D_k}^Q \given \boF_{k+1})}_{L^\infty}.
  \end{equation}
For $r\in \{2, Q\}$, we have $\E(\abs{D_k}^r \given \boF_{k+1})(y) =
\boL(\abs{A_k}^r)(T^{k+1}y)$. This implies $\sum \norm{\E(\abs{D_k}^r
\given \boF_{k+1})}_{L^\infty} =
\norm{\boL(\abs{A_k}^r)}_{L^\infty}$.

Consider a point $x\in X$. If it does not belong to $\Delta_0$, it
has a unique preimage $z$, and moreover $A_k(z)=0$. Hence,
$\boL(\abs{A_k}^r)(x)=0$. Suppose now $x\in \Delta_0$. Let $z_\alpha$
denote its preimages (with respective heights $h_\alpha-1$).
Lemma~\ref{lem_Fk} gives $A_k(z_\alpha) = O(1+h_\alpha\wedge k)$.
Hence,
  \begin{equation*}
  \boL(\abs{A_k}^r)(x) = \sum_\alpha g(z_\alpha) \abs{A_k(z_\alpha)}^r
  \leq C + \sum g(z_\alpha) (h_\alpha\wedge k)^r
  \leq C \int_{\Delta_0} (\phi \wedge k)^r \dd\mu.
  \end{equation*}
We have proved that
  \begin{equation}
  \label{eq:controle_Ak2}
  \norm{\E(\abs{D_k}^r \given \boF_{k+1})}_{L^\infty} \leq C\int_{\Delta_0} (\phi \wedge k)^r \dd\mu.
  \end{equation}
We use this inequality to estimate the two sums on the right hand
side of~\eqref{eq:BR}. For $r=2$, the above integral is uniformly
bounded since $\phi$ has a moment of order $2$. Hence, the first sum
in~\eqref{eq:BR} is bounded by $n^{Q/2}=n^{q-1}$. For $r=Q=2q-2$, the
above integral is bounded by $k^{q-2}$ thanks to
Lemma~\ref{lem:karamata}. Summing over $k$, it follows that the
second sum in~\eqref{eq:BR} is bounded by $n^{q-1}$, as desired.
\qed

\subsection{The case \texorpdfstring{$q<2$}{q<2}}

In this paragraph, we prove Theorem~\ref{main_thm:moments} in the
case $q\in (1,2)$. Again, it suffices to prove the
estimate~\eqref{eq:weak_q} regarding the weak $q$-moment, i.e.,
$\norm{S_n f}_{L^{q,w}} \leq C n^{1/q}$, since the other estimates
follow using~\eqref{eq:deduce_moments}.

We start again from the decomposition $S_n f=\sum A_k\circ T^k +
\E(S_n f \given \boF_n)$. We rely on the von Bahr-Esseen result
for weak moments given in Theorem~\ref{prop:weak_vBE}, for $Q=q$.

First, we control the last term, as above: we have $\norm{\E(S_nf
\given \boF_n)}_{L^q} \leq \sum_{k=1}^n \norm{\boL^k f}_{L^q}$.
Moreover, we have as above $\norm{\boL^k f}_{L^q} \leq
C/k^{(q-1)/q}$. Summing over $k$,
  \begin{equation*}
  \norm{\E(S_nf \given \boF_n)}_{L^q} \leq C\sum_{k=1}^n k^{1/q-1}\leq C n^{1/q}.
  \end{equation*}
As the weak $L^q$-norm is dominated by the strong $L^q$-norm, this is
the desired control.

Now, we turn to the contribution of $A_k$. We want to estimate
$\norm{A_k}_{L^{q,w}}$. If $Tx\not\in \Delta_0$, then $A_k(x)=0$. If
$Tx\in \Delta_0$, then $\abs{A_k(x)} \leq C(1+k\wedge h)\leq C(1+h)$
by Lemma~\ref{lem_Fk}. Hence, for $s$ larger than a fixed constant,
  \begin{equation*}
  \mu\{x \st \abs{A_k(x)} \geq s\} \leq C \mu\{y\in \Delta_0 \st \phi(y) \geq C^{-1} s\}=O(s^{-q}).
  \end{equation*}
This shows that $\norm{A_k}_{L^{q,w}}$ is uniformly bounded. Summing
over $k$ and using Theorem~\ref{prop:weak_vBE}, we get $\norm{\sum
A_k\circ T^k}^q_{L^{q,w}} \leq Cn$, as desired.
\qed

\subsection{The case \texorpdfstring{$q=2$}{q=2}}

In this paragraph, we prove Theorem~\ref{main_thm:moments} in the
case $q=2$. Contrary to the previous cases, it is not sufficient to
prove the result at the critical exponent $p=2$, one should also
control all $p>2$. The arguments in the proof of the case $q>2$
(notably Burkholder's inequality~\eqref{eq:BR} combined
with~\eqref{eq:controle_Ak2}) give, for a general $p\geq 2$,
  \begin{equation}
  \label{eq:q=2}
  \norm{S_n f}_{L^p}^p \leq C \left(\sum_{k=0}^{n-1} \int_{\Delta_0}(\phi \wedge k)^2\right)^{p/2} \dd\mu
  + C \sum_{k=0}^{n-1}\int_{\Delta_0}(\phi \wedge k)^p \dd\mu
  + \left(\sum_{k=1}^n \norm{\boL^k f}_{L^p}\right)^p.
  \end{equation}
First, we have $\int \abs{\boL^k f}^p \dd\mu\leq C/k$ since the speed
of decay of correlations is $1/k$. Hence, $\norm{\boL^k f}_{L^p} \leq
k^{-1/p}$ and the last term in~\eqref{eq:q=2} is bounded by
$n^{p-1}$.

Let us now deal with $p=2$. Lemma~\ref{lem:karamata} gives
$\int_{\Delta_0}(\phi \wedge k)^2 \dd\mu \leq \log k$ since we are
precisely at the critical exponent for which there is an additional
logarithmic factor. Summing over $k$ and using~\eqref{eq:q=2}, we
obtain $\norm{S_n f}_{L^2}^2 \leq n \log n$ as desired.

Consider then $p>2$. Again, $\int_{\Delta_0}(\phi \wedge k)^2 \dd\mu
\leq \log k$, hence the first sum in~\eqref{eq:q=2} gives a
contribution $C(n\log n)^{p/2}$, which is bounded by $C n^{p-1}$ as
$p/2<p-1$. For the second sum in~\eqref{eq:q=2},
Lemma~\ref{lem:karamata} gives $\int (\phi \wedge k)^p \dd\mu \leq C
k^{p-2}$. Summing over $k$, we get a bound $n^{p-1}$.
\qed

\section{Concentration bounds}
\label{sec:concentration}

In this section, we prove Theorem~\ref{main_thm:concentration} about
concentration inequalities in Young towers with $\tail_n=O(n^{-q})$
for some $q>1$. As before, we write $c_n^{(q)}$ for a generic
sequence that is $O(n^{-q})$.

Consider a general function $K(x_0,x_1,\dotsc)$ which is separately
Lipschitz in each variable, with corresponding constants $\Lip_i(K)$.
Fix any reference point $x_*$ in the space.

To study the magnitude of $K(x,Tx,\dotsc)$, the idea is to decompose
it as a sum of reverse martingale differences. We consider $K$ as a
function defined on the space $\tilde X=X^{\N}$, endowed with the
probability measure $\tilde\mu=\mu\otimes
\delta_{Tx}\otimes\delta_{T^2x}\otimes \dotsm$. Let $\boF_k$ be the
$\sigma$-algebra generated by indices starting with $k$ (i.e., a
function $f(x_0,x_1,\dotsc)$ on $\tilde X$ is $\boF_k$-measurable if
it does not depend on $x_0,\dotsc, x_{k-1}$). Let
  \begin{equation}
  \label{eq:def_Kk}
  K_k(x_k,\dotsc) = \E(K \given \boF_k)(x_k,\dotsc)
  = \sum_{T^k x=x_k} g^{(k)}(x) K(x,\dotsc, T^{k-1}x, x_k,\dotsc).
  \end{equation}
This function plays the role of the function $F_k$ (defined
after~\eqref{eq:Ak}) for Birkhoff sums, and is the main object to
understand.

As in the proof of Lemma~\ref{lem_Fk}, we want to express
$K_k(x_k,\dotsc)$, for $x_k\in \Delta_0$, using the transfer operator
restricted to the basis, i.e., $T_n$. Define for $i\leq k$ a function
$w_i$ on the basis by
  \begin{align*}
  w_i(x) = \sum_{T^i y = x} g^{(i)}(y) \Bigl[&
    K(y, Ty,\dotsc, T^{j(y)-1}y, T^{j(y)}y,\dotsc,
       T^{i-1}y, x, \underbrace{x_*,\dotsc, x_*}_{k-i-1\text{ terms}}, x_k, \dotsc)
    \\
    -& K(y, Ty, \dotsc, T^{j(y)-1}y, T^{j(y)}y, \underbrace{x_*,\dotsc, x_*}_{\mathclap{k-j(y)-1\text{ terms}}}, x_k,\dotsc)\Bigr],
  \end{align*}
where for each $y$ we define $j(y)$ as the last time in $[0, i-1]$
for which $T^{j(y)}(y) \in \Delta_0$. If there is no such time, then
$j(y)=-1$. The idea is that, for each preimage $y$ of $x$ under
$T^i$, we replace its last excursion outside of $\Delta_0$ by the
trivial sequence $x_*,\dotsc, x_*$.

A simple telescoping argument then gives:
  \begin{equation*}
  K_k(x_k,\dotsc) = \sum_{i=0}^k \sum_{x\in \Delta_0, T^{k-i}x=x_k} g^{(k-i)}(x) w_i(x) + K(x_*,\dotsc, x_*,x_k,\dotsc).
  \end{equation*}
Indeed, in the expression~\eqref{eq:def_Kk}, if one starts replacing
successively each excursion outside of $\Delta_0$, one ends up adding
sums of the functions $w_i(x)$, and the remaining term (where all
excursions have been replaced) is $\sum_{T^k x =x_k} g^{(k)}(x)
K(x_*,\dotsc, x_*, x_k,\dotsc)$, which reduces to $K(x_*,\dotsc, x_*,
x_k,\dotsc)$ since $\sum g^{(k)}(x)=1$ as the measure is invariant.

The above expression also reads
  \begin{equation}
  \label{eq:decompose_Kk}
  K_k(x_k,\dotsc) = \sum_{i=0}^k T_{k-i}w_i (x_k) + K(x_*,\dotsc, x_*, x_k, \dotsc).
  \end{equation}
We will be able to use it since we know a lot about the operators
$T_n$ (their properties, expressed in
Proposition~\ref{prop:conseq_banach}, were already at the heart of
the proof of Lemma~\ref{lem_Fk}), but we first need to understand
$w_i$ more properly.

\begin{lem}
We have $\norm{w_i}\leq \sum_{a+b=i} \Lip_a(K) c_{b}^{(q)}$.
\end{lem}
\begin{proof}
First, we control the supremum of $w_i$. Write $w_i(x)=\sum_y
g^{(i)}(y) H(y)$, then $\abs{H(y)} \leq \sum_{j(y)+1}^{i-1}
\Lip_\ell(K)$. The sum of $g^{(i)}(y)$ over those points with
$j(y)<\ell$ is $\sum_{T^{i-\ell}z=x} g^{(i-\ell)}(z)$, where the sum
is restricted to those points $z$ that do not come back to the basis
before time $i-\ell$. By bounded distortion, this is comparable to
$\mu\{\phi>i-\ell\} \leq c_{i-\ell}^{(q)}$. We get
  \begin{equation}
  \label{eq:wi_sup}
  \norm{w_i}_\infty \leq \sum_\ell \Lip_\ell(K) c_{i-\ell}^{(q)}.
  \end{equation}

We estimate now the Lipschitz constant of $w_i$. Write for $x,x'\in
\Delta_0$
  \begin{align*}
  w_i(x)-w_i(x') & = \sum g^{(i)}(y) H(y) -g^{(i)}(y')H(y')
  \\&
  = \sum g^{(i)}(y) (H(y)-H(y')) + \sum (g^{(i)}(y)-g^{(i)}(y')) H(y'),
  \end{align*}
where we have paired together the preimages $y$ and $y'$ of $x$ and
$x'$ under $T^i$ that belong to the same cylinder of length $i$. For
the second sum, bounded distortion gives $\abs{g^{(i)}(y) -
g^{(i)}(y')} \leq C d(x,x') g^{(i)}(y') $, hence the Lipschitz norm
of this sum is at most $C \norm{w_i}_\infty$, which has already been
controlled in~\eqref{eq:wi_sup}. For the first sum, we have
  \begin{equation*}
  \abs{H(y)-H(y')} \leq 2\sum_{\ell=0}^{i-1} \Lip_\ell(K) d(T^\ell y, T^\ell y')
  \leq 2 \sum_{\ell=0}^{i-1} \Lip_\ell(K)\Psi_{i-\ell}(T^\ell y)d(x,x'),
  \end{equation*}
where $\Psi_{a}(z) = \rho^{\Card\{0\leq t< a, T^t z\in \Delta_0\}}$:
this function measures the expansion of the map $T^a$ applied to $z$,
since each return to the basis gives an expansion factor of
$\rho^{-1}>1$ by definition of the distance. Using bounded
distortion, we get
  \begin{equation*}
  \abs*{\sum g^{(i)}(y) (H(y)-H(y'))}  \leq C \sum_{\ell=0}^{i-1} \Lip_\ell(K) d(x,x')
    \int_{T^{-(i-\ell)}\Delta_0} \Psi_{i-\ell}\dd\mu.
  \end{equation*}
By~\cite[Lemma 4.4]{gouezel_chazottes_concentration}, the sequence
$\int_{T^{-n}\Delta_0} \Psi_{n}\dd\mu$ is $\leq c_n^{(q)}$. The
desired bound for the Lipschitz constant of $w_i$ follows.
\end{proof}

Then, we turn to the analogue of Lemma~\ref{lem_Fk}.
\begin{lem}
\label{lem:Kk}
If $x_k$ is at height $h$ and $x_{k+1}\in \Delta_0$, then
  \begin{multline}
  \label{eq_avec_decroissance}
  \abs{K_k(x_k, x_{k+1}, x_{k+2},\dotsc)-K_{k+1}(x_{k+1}, x_{k+2},\dotsc)}
  \\
  \leq \sum_{a=0}^{k-h} \Lip_a(K)\min\left( (h+1)c_{k-h-a}^{(q)}, c_{k-h-a}^{(q-1)}\right)
  + \sum_{a=k-h+1}^k \Lip_a(K).
  \end{multline}
\end{lem}
When $h>k$, the first sum vanishes, and the second one reduces to
$\sum_{a=0}^k \Lip_a(K)$ since $\Lip_a(K)=0$ for $a<0$.

If all the $\Lip_a(K)$ are of order $1$ (which is the case for
instance with Birkhoff sums), it is easy to check that the expression
in the lemma reduces to $O(1+h\wedge k)$ as in Lemma~\ref{lem_Fk}.

\begin{proof}
The case $h>k$ is easy (just substitute each variable in the
expression of $K_k(x_k,\dotsc)$ with the corresponding variable in
$K_{k+1}(x_{k+1},\dotsc)$), let us deal with the more interesting
case $h\leq k$.

We first prove the inequality
  \begin{multline}
  \label{eq_sans_decroissance}
  \abs{K_k(x_k, x_{k+1}, x_{k+2},\dotsc)-K_{k+1}(x_{k+1}, x_{k+2},\dotsc)}
  \\
  \leq \sum_{a=0}^{k-h} \Lip_a(K)\left[
    \sum_{b=0}^{k-h-a} c_b^{(q)} \left(
      \sum_{j=k-h-a-b}^{k-a-b} c_j^{(q)}\right)
    +\sum_{b=k-h-a+1}^{k-a+1} c_b^{(q)}\right]
  + \sum_{a=k-h+1}^k \Lip_a(K).
  \end{multline}

We replace successively all the variables with index in $(k-h,k]$ in
the expressions of $K_k(x_k,\dotsc)$ and $K_{k+1}(x_{k+1},\dotsc)$
with $x_*$, introducing an error at most $\sum_{a=k-h+1}^k \Lip_a(K)$
that corresponds to the last term in~\eqref{eq_sans_decroissance}.
Letting
  \begin{equation*}
  \tilde K(x_0,\dotsc, x_{k-h}) = K(x_0,\dotsc, x_{k-h},x_*,\dotsc, x_*, x_{k+1},\dotsc)-K(x_*,\dotsc, x_*, x_{k+1},\dotsc),
  \end{equation*}
we may then work with $\tilde K$ instead of $K$. It satisfies
$\Lip_a(\tilde K)\leq \Lip_a(K)$ for $a\leq k-h$, and $\Lip_a(\tilde
K)=0$ for $a>k-h$. Let $w_i$ be the corresponding functions for
$\tilde K$, and let $x=\pi x_k$ be the projection of $x_k$ in the
basis of the tower. We get from~\eqref{eq:decompose_Kk}
  \begin{equation*}
  \tilde K_k(x_k) = \tilde K_{k-h}(x) = \sum_{i=0}^{k-h}T_{k-h-i} w_i(x),\quad
  \tilde K_{k+1}(x_{k+1}) = \sum_{i=0}^{k+1} T_{k+1-i} w_i(x_{k+1}).
  \end{equation*}
We write $T_\ell = \Pi T_\ell \Pi + E_\ell$, where $\Pi$ is the
projection on constant functions, and $\norm{E_\ell} \leq
c_{\ell}^{(q)}$ by Proposition~\ref{prop:conseq_banach}. We have
  \begin{align*}
  \abs*{\sum_{i=0}^{k-h}E_{k-h-i} w_i(x)} &
  \leq \sum_{\ell+i=k-h} \norm{E_\ell} \norm{w_i}
  \leq \sum_{\ell+ i=k-h} c_\ell^{(q)} \sum_{a+b=i} \Lip_a(\tilde K) c_b^{(q)}
\\&
  =\sum_{a+j=k-h} \Lip_a(\tilde K) (c^{(q)}\star c^{(q)})_j.
  \end{align*}
By~\eqref{seq:convole}, this is bounded by $\sum_{a+j=k-h} \Lip_a(K)
c_j^{(q)}$, which is bounded by~\eqref{eq_sans_decroissance} (to see
this, in~\eqref{eq_sans_decroissance}, take $b=0$ in the first sum
over $b$, and then $j=k-h-a$ in the next sum). In the same way, we
have
  \begin{equation*}
  \abs*{\sum_{i=0}^{k+1}E_{k+1-i} w_i(x_{k+1})}
  \leq \sum_{a+j=k+1} \Lip_a(\tilde K) c_j^{(q)},
  \end{equation*}
which is again bounded by~\eqref{eq_sans_decroissance} (up to a shift
of one in the indices, take $b=0$ in the first sum
of~\eqref{eq_sans_decroissance} and $j=k-a$ in the second sum).

We turn to the main terms, coming from $\Pi T_\ell \Pi$. We have $\Pi
T_\ell \Pi w_i = t_\ell u_i$, for some scalar sequences $t_\ell$ and
$u_i$. Moreover, $u_i$ is bounded by $\norm{w_i}$, and
$\abs{t_\ell-t_{\ell+1}}\leq c_\ell^{(q)}$ by
Proposition~\ref{prop:conseq_banach}. The resulting term is
  \begin{equation*}
  \abs*{\sum_{i=0}^{k-h} t_{k-h-i}u_i - \sum_{i=0}^{k+1} t_{k+1-i} u_i}
  \leq \sum_{i=0}^{k-h} \abs{u_i} \sum_{j=k-h-i}^{k-i} \abs{t_{j+1}-t_j}
  + \sum_{i=k-h+1}^{k+1} \abs{u_i}.
  \end{equation*}
Bounding $\abs{u_i}$ by $\sum_{a+b=i}\Lip_a(\tilde K) c_b^{(q)}$ and
$\abs{t_{j+1}-t_j}$ with $c_j^{(q)}$, we readily check that all those
terms are bounded by~\eqref{eq_sans_decroissance}.

This concludes the proof of~\eqref{eq_sans_decroissance}. To
conclude, we should show that the coefficient of $\Lip_a(K)$ in this
expression is bounded by $\min\left((h+1)c_{k-h-a}^{(q)},
c_{k-h-a}^{(q-1)}\right)$. We have $\sum_{i=\ell}^\infty c_i^{(q)}
\leq c_\ell^{(q-1)}$. In particular,
  \begin{align*}
  \sum_{b=0}^{k-h-a}c_b^{(q)}\left(\sum_{j=k-h-a-b}^{k-a-b} c_j^{(q)}\right)
  &
  \leq \sum_{b=0}^{k-h-a} c_b^{(q)} c_{k-h-a-b}^{(q-1)}
  = (c^{(q)}\star c^{(q-1)})_{k-h-a}
  \\&
  \leq (c^{(q-1)}\star c^{(q-1)})_{k-h-a}
  \leq c^{(q-1)}_{k-h-a},
  \end{align*}
as the sequences that are $O(1/n^{q-1})$ are stable under
convolution. This proves the upper bound $c^{(q-1)}_{k-h-a}$. For the
other one, note that $\sum_{j=k-h-a-b}^{k-a-b}c_j^{(q)} \leq
(h+1)c_j^{(q)}$. From this point on, one can continue the computation
as above, getting in the end the bound $(h+1)c^{(q)}_{k-h-a}$.
\end{proof}

\begin{rmk}
The article~\cite{gouezel_chazottes_concentration} already proved
concentration estimates in Young towers, but only for $q>2$. In this
case, the estimates were not as good as those in
Theorem~\ref{main_thm:concentration}. Moreover, all the estimates
started diverging when $q\leq 2$. There are three main differences in
the current approach that make it possible to improve
upon~\cite{gouezel_chazottes_concentration}:
\begin{itemize}
\item The decomposition~\eqref{eq:decompose_Kk} of $K_k$, where
    one replaces one excursion at a time in the definition of
    $w_i$, is more efficient than the corresponding decomposition
    of~\cite{gouezel_chazottes_concentration} where one only
    replaces one variable at a time (this creates some useless
    redundancy in the estimates, which is not a problem when
    $q>2$ but causes divergence of the estimates when $q\leq 2$).
\item The main difference between the current paper
    and~\cite{gouezel_chazottes_concentration} is that, in
    Lemma~\ref{lem:Kk}, we compare directly $K_k$ to $K_{k+1}$.
    On the contrary, in~\cite{gouezel_chazottes_concentration},
    $K_k$ and $K_{k+1}$ are compared to explicit integral
    quantities (see for instance Lemma~2.3 there). This is more
    intuitive and natural, since it expresses the mixing
    properties of the system. However, when $q\leq 2$, the
    convergence towards these integrals is rather slow, making
    again the estimates diverge. In the proof of
    Lemma~\ref{lem:Kk}, we do \emph{not} claim that $K_k$ is
    close to any explicit or meaningful quantity, only that it is
    close to $\sum t_{k-i} u_i$. This is sufficient to prove that
    $K_k$ is close to $K_{k+1}$ since $t_n$ is close to $t_{n+1}$
    by Proposition~\ref{prop:conseq_banach}. Both are also close
    to $\lim t_i$ if $n$ is large enough, and this is essentially
    what is used in~\cite{gouezel_chazottes_concentration}, but
    this gives a worse estimate.
\item In the case $q>2$, the main new ingredient compared to the
    techniques of~\cite{gouezel_chazottes_concentration} is
    Lemma~\ref{lem:ineqq>2} below.
\end{itemize}
\end{rmk}

We can now deduce concentration bounds in the different situations we
considered for moment bounds.

\subsection{The case \texorpdfstring{$q>2$}{q>2}}

In this paragraph, we prove Theorem~\ref{main_thm:concentration} in
the case $q>2$. As we explained after the statement of this theorem,
it suffices to prove the result for $p=2q-2$.

In this situation, we use~\eqref{eq_avec_decroissance} in the
form
  \begin{equation}
  \label{lem:Kk_q>2}
  \abs{K_k(x_k, x_{k+1},\dotsc)-K_{k+1}(x_{k+1},\dotsc)}
  \leq \sum_{a=0}^{k-h}\Lip_a(K) c_{k-h-a}^{(q-1)}
  + \sum_{a=k-h+1}^k \Lip_a(K),
  \end{equation}
i.e., we always use the same term $c_{k-h-a}^{(q-1)}$ in the minimum
in~\eqref{eq_avec_decroissance}.

Let us start the proof of the theorem. The quantity $K-\E(K)$ can be
decomposed as $\sum_{k\geq 0} (K_k-K_{k+1})$. Since this is a sum of
reverse martingale differences, we may use Burkholder-Rosenthal
inequality in the form of~\eqref{eq:burkh_useful}, to obtain a bound
  \begin{multline*}
  \int \abs{K(x,Tx,\dotsc)-\E(K)}^{2q-2}\dd\mu(x)
  \\
  \leq C \left(\sum \norm{\E(D_k^2 \given \boF_{k+1})}_{\infty}\right)^{q-1}
  + C \sum \norm{\E(\abs{D_k}^{2q-2} \given \boF_{k+1})}_{\infty},
  \end{multline*}
where $D_k=K_k-K_{k+1}$. Hence, for $r \in \{2, 2q-2\}$, we should
estimate $\norm{\E(\abs{D_k}^r \given \boF_{k+1})}_{\infty}$. If
$x_{k+1}$ is not in the basis of the tower, then $\E(\abs{D_k}^r
\given \boF_{k+1})(x_{k+1},\dotsc)=0$ and there is nothing to do.
Assume now that $x_{k+1}$ is in the basis. Let $z_\alpha$ denote its
preimages, with respective heights $h_\alpha$. We have
  \begin{equation*}
  \E(\abs{D_k}^r\given \boF_{k+1})(x_{k+1},\dotsc) =
  \sum g(z_\alpha) \abs{K_k(z_\alpha, x_{k+1},\dotsc)-K_{k+1}(x_{k+1},\dotsc)}^r.
  \end{equation*}
With~\eqref{lem:Kk_q>2}, we get
  \begin{equation*}
  \norm{\E(\abs{D_k}^r \given \boF_{k+1})}_\infty
  \leq
  \sum_{h\geq 0} \mu(\phi=h+1) \left( \sum_{a=0}^{k-h} \Lip_a(K) c_{k-h-a}^{(q-1)} + \sum_{a=k-h+1}^k \Lip_a(K)\right)^r
  \end{equation*}
Using the inequality $(X+Y)^r \leq C(X^r+Y^r)$ to separate the two
sums, we get two different terms. We should then sum over $k$, and
get a bound in terms of $\sum \Lip_a(K)^2$.

First, we deal with the first sum $\sum_{a=0}^{k-h} \Lip_a(K)
c_{k-h-a}^{(q-1)}$. Since the sequence $c_n^{(q-1)}$ is summable, we
have by~\eqref{eq:stable_power}
  \begin{equation*}
  \left(\sum_{a=0}^{k-h} \Lip_a(K) c_{k-h-a}^{(q-1)}\right)^r
  \leq C \sum_{a=0}^{k-h} \Lip_a(K)^r c_{k-h-a}^{(q-1)}.
  \end{equation*}
Summing over $k$, we get a term
  \begin{equation*}
  \sum_k \sum_{h=0}^k \sum_{a=0}^{k-h} \mu(\phi=h+1) \Lip_a(K)^r c_{k-h-a}^{(q-1)}.
  \end{equation*}
Writing $k=\ell+h$, this becomes
  \begin{equation*}
  \sum_{h\geq 0} \sum_{\ell\geq 0} \sum_{a=0}^{\ell} \mu(\phi=h+1) \Lip_a(K)^r c_{\ell-a}^{(q-1)}.
  \end{equation*}
The sum over $h$ factorizes out. Then, for each $a$, the sum over
$\ell$ gives a finite contribution since $c^{(q-1)}$ is summable. We
are left with $\sum_a \Lip_a(K)^r$, which is bounded by $\left(\sum_a
\Lip_a(K)^2\right)^{r/2}$ as desired.

Then, we deal with the second sum $\sum_{a=k-h+1}^k \Lip_a(K)$.
Summing over $k$, the corresponding term is
  \begin{equation*}
  \sum_k \sum_h \mu(\phi=h+1) \left(\sum_{a=k-h+1}^k \Lip_a(K)\right)^r.
  \end{equation*}
We need to treat separately the cases $r=2$ and $r=2q-2$. For $r=2$,
we simply use Cauchy-Schwarz inequality:
  \begin{align*}
  \sum_k \sum_h \mu(\phi=h+1) \left(\sum_{a=k-h+1}^k \Lip_a(K)\right)^2
  &
  \leq \sum_k \sum_h \mu(\phi=h+1) h \sum_{a=k-h+1}^k \Lip_a(K)^2
  \\&
  = \sum_a \Lip_a(K)^2 \sum_h\mu(\phi=h+1) h\sum_{k=a}^{a+h-1}1.
  \end{align*}
We can factorize out $\sum_h \mu(\phi=h+1)h^2$, which is finite since
$q>2$, by Lemma~\ref{lem:karamata}. We are left with $C\sum_a
\Lip_a(K)^2$ as desired.

For the case $r=2q-2$, we should prove an inequality
  \begin{equation*}
  \sum_k \sum_h \mu(\phi=h+1) \left(\sum_{a=k-h+1}^k \Lip_a(K)\right)^{2q-2}
  \leq C \left(\sum_a \Lip_a(K)^2\right)^{q-1}.
  \end{equation*}
It turns out that this inequality is more difficult than the previous
ones. It is given in Lemma~\ref{lem:ineqq>2} below. With this lemma,
the proof is complete.
\qed

\begin{lem}
\label{lem:ineqq>2}
Let $q>2$. Consider a sequence $a_n\geq 0$ with $\sum_{n\geq N}a_n =
O(N^{-q})$. There exists a constant $C$ such that, for any sequence
$(u_n)\in \ell^2(\Z)$,
  \begin{equation*}
  \sum_{n\in \Z} \sum_{h\geq 0} a_h \left|\sum_{i=n-h}^{n+h} u_i\right|^{2q-2}
  \leq C \left(\sum_{n\in \Z} u_n^2\right)^{q-1}.
  \end{equation*}
\end{lem}
Although the statement of the lemma is completely elementary, this
result is not trivial, even for $a_n=1/n^{q+1}$ (as is maybe
indicated by the fact that it fails for $q=2$). In particular, we
have not been able to find a direct proof: We need to resort to
maximal inequalities and interpolation.
\begin{proof}
We associate to a sequence $u_n$ the sequence
  \begin{equation*}
  v(n,h) = \frac{\sum_{i=n-h}^{n+h} u_i}{(2h+1)^{1/2}}.
  \end{equation*}
We consider $v$ as a function on the space $\Z\times \N$ endowed with
the measure $\nu(n,h) = (h+1)^{q-1} a_h$.

By Cauchy-Schwarz inequality, the function $v$ is bounded in
$L^\infty$ by $\norm{u}_{\ell^2}$. Let us now consider its weak
$L^2$-norm. Let $Mu(n) = \sup_h \left(\sum_{i=n-h}^{n+h}
\abs{u_i}\right)/(2h+1)$ be the maximal function associated to $u$.
It is bounded in $\ell^2$ by $C \norm{u}_{\ell^2}$, by
Theorem~\ref{thm:maximal}. Since $v(n,h) \leq (2h+1)^{1/2} Mu(n)$, we
have for all $s\geq 0$
  \begin{align*}
  \nu\{(n,h) \st \abs{v(n,h)}\geq s \}
  &
  \leq \nu\{(n,h) \st (2h+1)^{1/2} \geq s/Mu(n)\}
  \\&
  \leq \sum_n \sum_{h \geq ((s/Mu(n))^2 - 1)/2} (h+1)^{q-1} a_h.
  \end{align*}
By Lemma~\ref{lem:karamata}, we have $\sum_{h\geq t} (h+1)^{q-1}a_h
\leq C/(t+1)$. Hence,
  \begin{equation*}
   \nu\{(n,h) \st \abs{v(n,h)}\geq s \} \leq C \sum_n Mu(n)^2/s^2 \leq C s^{-2} \norm{Mu}_{\ell^2}^2
   \leq C s^{-2} \norm{u}_{\ell^2}^2.
  \end{equation*}

This shows that $v$ is bounded in $L^\infty$ and in weak $L^2$ by
$C\norm{u}_{\ell^2}$. One could deduce boundedness in any $L^p$ for
$2<p<\infty$ by using classical interpolation arguments, but it is
simpler to use the formula~\eqref{eq:deduce_moments}: we get
  \begin{equation}
  \label{eq:absvp}
  \int \abs{v}^p \dd\nu \leq C \int_{s=0}^{\norm{v}_\infty} s^{p-1} Cs^{-2} \norm{u}_{\ell^2}^2 \dd s
  \leq C \norm{u}_{\ell^2}^2 \norm{v}_\infty^{p-2}
  \leq C \norm{u}_{\ell^2}^p.
  \end{equation}

Taking $p=2q-2$, we get
  \begin{equation*}
  C \left(\sum u_n^2\right)^{q-1} \geq \sum_n \sum_h \nu(n,h) \abs{v(n,h)}^{2q-2}
  = \sum_n \sum_h (h+1)^{q-1} a_h \abs*{\frac{\sum_{i=n-h}^{n+h} u_i}{(2h+1)^{1/2}}}^{2q-2}.
  \end{equation*}
The powers of $h$ cancel on the right, and we are left with the
statement of the lemma.
\end{proof}

\subsection{The case \texorpdfstring{$q<2$}{q<2}}

In this case, it is sufficient to prove the weak moment
estimate~\eqref{eq:concent_q<2_weak}, since it implies all the other
ones thanks to~\eqref{eq:deduce_moments}. Let us for instance explain
how to get the most complicated moment estimate, for $p=q$. Write
$A=\sum \Lip_i(K)^q$, so that $\mu\{\abs{K-\E K}\geq s\}\leq
CAs^{-q}$, and $B = \sum \Lip_i(K) \geq \norm{K-\E K}_{L^\infty}$.
Then
  \begin{align}
  \notag
  \int \abs{K-\E K}^q \dd\mu
  &= q \int s^{q-1} \mu\{\abs{K-\E K}\geq s\} \dd s
  \leq q\int_{s=0}^B s^{q-1}\min(1, CAs^{-q})\dd s
  \\&
  \label{eq:deduce_strong_from_weak}
  \leq q\int_{s=0}^{A^{1/q}} s^{q-1}\dd s
  + q \int_{s=A^{1/q}}^B s^{q-1} CAs^{-q} \dd s
  = A + qCA(\log B-\log A^{1/q}).
  \end{align}
This is the desired moment estimate.

Let us now start the proof of~\eqref{eq:concent_q<2_weak}. Thanks to
Proposition~\ref{prop:weak_vBE}, the decomposition $K-\E K = \sum
D_k$ (with $D_k=K_k-K_{k+1}$) gives
  \begin{equation*}
  \norm{K-\E K}_{L^{q,w}}^q \leq C \sum \norm{D_k}_{L^{q,w}}^q.
  \end{equation*}
We have $D_k(x)=0$ if $Tx\not\in \Delta_0$, and otherwise
Lemma~\ref{lem:Kk} gives the bound
  \begin{equation*}
  \abs{D_k(x)} \leq \sum_{a=0}^{k-h} \Lip_a(K)\min\left( (h+1)c_{k-h-a}^{(q)}, c_{k-h-a}^{(q-1)}\right)
  + \sum_{a=k-h+1}^k \Lip_a(K),
  \end{equation*}
where $h=h(x)$. We should bound the weak $L^q$ norm of both terms on
the right to conclude. Let us denote them by $U_k(x)$ and $V_k(x)$.

We start with $V_k$. Fix some $s\geq 0$, let $h_0(k)$ be minimal such
that $\sum_{a=k-h_0+1}^k \Lip_a(K) \geq s$. Then
  \begin{equation*}
  \mu\left\{x \in T^{-1}\Delta_0 \st \sum_{a=k-h+1}^k \Lip_a(K) \geq s\right\}
  =\mu\{x\in T^{-1}\Delta_0 \st h(x)\geq h_0\}.
  \end{equation*}
This measure is exactly $\mu\{y\in \Delta_0 \st
\phi(y)>h_0\}=O(h_0^{-q})$. Hence,
  \begin{align*}
  \mu\left\{x \st \sum_{a=k-h+1}^k \Lip_a(K) \geq s\right\}
  &\leq C (1+h_0)^{-q}
  \leq C (1+h_0)^{-q} \left(\sum_{a=k-h_0+1}^k \Lip_a(K)/s\right)^q
  \\&
  \leq C s^{-q} M(k)^q,
  \end{align*}
where $M(k)$ is the maximal function associated to $\Lip_i(K)$, i.e.,
  \begin{equation}
  \label{def_Mk}
  M(k) = \sup_{h\geq 0}\frac{1}{2h+1} \sum_{i=k-h}^{k+h} \Lip_i(K).
  \end{equation}
We have proved that $\norm{V_k}_{L^{q,w}}^q \leq C M(k)^q$. Summing
over $k$, we obtain
  \begin{equation*}
  \sum \norm{V_k}_{L^{q,w}}^q \leq C \sum_k M(k)^q \leq C \sum \Lip_i(K)^q,
  \end{equation*}
since $M$ is bounded in $\ell^q(\Z)$ by
$C\norm{\Lip_i(K)}_{\ell^q(\Z)}$, by Theorem~\ref{thm:maximal}. This
is the desired upper bound.

We turn to $U_k$. We have
  \begin{align*}
  \sum_k \norm{U_k}_{L^{q,w}}^q
  &\leq \sum_k \norm{U_k}_{L^q}^q
  \\&
  \leq \sum_k \sum_h \mu(\phi=h+1) \left(\sum_{a=0}^{k-h} \Lip_a(K)\min\left( (h+1)c_{k-h-a}^{(q)}, c_{k-h-a}^{(q-1)}\right)\right)^q.
  \end{align*}
The next lemma shows that this is bounded by $C\sum \Lip_i(K)^q$ (set
$\epsilon=q-1$, $n=k-h$, $i=a$ and $u_i=\Lip_i(K)$ to reduce to this
statement). This concludes the proof.
\qed

\begin{lem}
\label{lem:technical_q<2}
Let $q>1$ and $\epsilon>0$. Consider a sequence $a_n\geq 0$ with
$\sum_{n\geq N}a_n = O(N^{-q})$. There exists a constant $C$ such
that, for any sequence $(u_n)\in \ell^q(\Z)$,
  \begin{equation*}
  \sum_{n\in \Z} \sum_{h\geq 0} a_h \left|\sum_{i\in \Z}u_i \cdot
     \min\left(\frac{h+1}{1+\abs{n-i}^{1+\epsilon}}, \frac{1}{1+\abs{n-i}^{\epsilon}}\right)\right|^q
  \leq C\sum_{n\in \Z} \abs{u_n}^q.
  \end{equation*}
\end{lem}
\begin{proof}
We proceed as in the proof of Lemma~\ref{lem:ineqq>2}. Define a
sequence
  \begin{equation*}
  v(n,h) = \frac{1}{(1+h)^{1-\epsilon}} \sum_{i\in \Z}u_i \cdot\min\left(\frac{h+1}{1+\abs{n-i}^{1+\epsilon}}, \frac{1}{1+\abs{n-i}^{\epsilon}}\right).
  \end{equation*}
We consider it as a function on the space $\Z\times \N$ with the
measure $\nu(n,h) = a_h (1+h)^{q(1-\epsilon)}$. We have for any $n\in
\Z$
  \begin{multline*}
  \sum_{i\in \Z} \min\left(\frac{h+1}{1+\abs{n-i}^{1+\epsilon}}, \frac{1}{1+\abs{n-i}^{\epsilon}}\right)
  \\
  \leq \sum_{\abs{m}\leq h} \frac{1}{1+\abs{m}^{\epsilon}} + (h+1)\sum_{\abs{m}>h} \frac{1}{1+\abs{m}^{1+\epsilon}}
  \leq C(1+h)^{1-\epsilon}.
  \end{multline*}
This shows that the operator $A:u\mapsto v$ is bounded from
$\ell^\infty(\Z, \mu)$ (where $\mu$ is the counting measure) to
$\ell^\infty(\Z\times \N, \nu)$. Moreover, writing $m=n-i$,
  \begin{align*}
  \sum_{n,h} \nu(n,h)& v(n,h)
  \\&
  = \sum_{n\in \Z} \sum_{h\geq 0} a_h (1+h)^{q(1-\epsilon)} \frac{1}{(1+h)^{1-\epsilon}}
    \sum_{i\in \Z}u_i \cdot\min\left(\frac{h+1}{1+\abs{n-i}^{1+\epsilon}}, \frac{1}{1+\abs{n-i}^{\epsilon}}\right)
  \\&
  = \sum_{i\in \Z} u_i \cdot \sum_{h\geq 0}  a_h (1+h)^{(q-1)(1-\epsilon)}
      \sum_{m\in \Z}\min\left(\frac{h+1}{1+\abs{m}^{1+\epsilon}}, \frac{1}{1+\abs{m}^{\epsilon}}\right).
  \end{align*}
As we have seen above, the last sum over $m$ is
$O((1+h)^{1-\epsilon})$. Hence, the sum over $h$ and $m$ reduces to
$\sum_h a_h (1+h)^{q(1-\epsilon)}$, which is finite by
Lemma~\ref{lem:karamata}. This shows that $\norm{v}_{\ell^1(\nu)}
\leq C \norm{u}_{\ell^1(\mu)}$.

The operator $A:u\mapsto v$ is bounded from $\ell^r(\mu)$ to
$\ell^r(\nu)$ for $r=1$ and $\infty$. By interpolation (see
Theorem~\ref{thm:interpolation}), it is also bounded from
$\ell^q(\mu)$ to $\ell^q(\nu)$. This is the desired inequality.
\end{proof}

\subsection{The case \texorpdfstring{$q=2$}{q=2}}

In this paragraph, we prove Theorem~\ref{main_thm:concentration} in
the case $q=2$. As we explained after the statement of this theorem,
it suffices to prove the result for $p\geq 2$. We follow essentially
the same steps as in the $q>2$ case. We start with
Burkholder-Rosenthal inequality~\eqref{eq:burkh_useful}
  \begin{multline*}
  \int \abs{K(x,Tx,\dotsc)-\E(K)}^{p}\dd\mu(x)
  \\
  \leq C \left(\sum \norm{\E(D_k^2 \given \boF_{k+1})}_{\infty}\right)^{p/2}
  + C \sum \norm{\E(\abs{D_k}^{p} \given \boF_{k+1})}_{\infty}.
  \end{multline*}
Moreover, for $r\in \{2, p\}$, we have
  \begin{align}
  \label{eq:q=2_D_k}
  \raisetag{-50pt}
  \begin{split}
   \norm{\E(\abs{D_k}^r \given \boF_{k+1})}_\infty
  \leq {} &
  C\sum_{h\geq 0} \mu(\phi=h+1) \left( \sum_{a=0}^{k-h} \Lip_a(K) \min\left((h+1)c_{k-h-a}^{(q)}, c_{k-h-a}^{(q-1)}\right)\right)^r
  \\&
  +C \sum_{h\geq 0} \mu(\phi=h+1) \left(\sum_{a=k-h+1}^k \Lip_a(K)\right)^r.
  \end{split}
  \end{align}

Let us first consider the contribution of the first line when we sum
over $k$. For $r=2$, Lemma~\ref{lem:technical_q<2} shows that the
resulting term is bounded by $C\sum \Lip_a(K)^2$. Its contribution to
Burkholder-Rosenthal inequality is therefore at most
  \begin{align*}
  C\left(\sum \Lip_i(K)^2 \right)^{p/2}
  &=C \left(\sum \Lip_i(K)^2 \right)\cdot \left(\sum \Lip_i(K)^2 \right)^{p/2-1}
  \\&\leq C \left(\sum \Lip_i(K)^2 \right) \left(\sum \Lip_i(K) \right)^{p-2},
  \end{align*}
since $\sum x_i^2 \leq (\sum x_i)^2$. This bound is compatible with
the statement of the theorem. For $r=p$, we write
  \begin{multline*}
  \left( \sum_{a=0}^{k-h} \Lip_a(K) \min\left((h+1)c_{k-h-a}^{(q)}, c_{k-h-a}^{(q-1)}\right)\right)^p
  \\\leq  \left( \sum_{a=0}^{k-h} \Lip_a(K) \min\left((h+1)c_{k-h-a}^{(q)}, c_{k-h-a}^{(q-1)}\right)\right)^2
    \cdot\left(\sum_{a\in \Z} \Lip_a(K)\right)^{p-2}.
  \end{multline*}
Using again Lemma~\ref{lem:technical_q<2}, it follows that the
contribution of this term to Burkholder-Rosenthal inequality is at
most $\left(\sum \Lip_i(K)^2 \right) \left(\sum \Lip_i(K)
\right)^{p-2}$ as desired.

Let us now turn to the second line of~\eqref{eq:q=2_D_k}. We define a
sequence
  \begin{equation*}
  v(k,h) = \sum_{a=k-h+1}^k \Lip_a(K)
  \end{equation*}
on the space $\Z\times \N$ with the measure $\nu(k,h)=\mu(\phi=h+1)$.
It satisfies $\norm{v}_{\ell^\infty} \leq \sum_{a\in \Z} \Lip_a(K)$.
Let us control its weak $L^2$ norm. Let $s\geq 0$. For fixed $k$, let
$h_0(k)$ be the smallest $h$ such that $\sum_{a=k-h+1}^k \Lip_a(K)
\geq s$. Then
  \begin{multline*}
  \nu\{(k,h)\st v(k,h)\geq s\} = \sum_k \sum_{h \geq h_0(k)} \mu(\phi=h+1)
  \leq C\sum_k (1+h_0(k))^{-2}
  \\
  \leq C\sum_k (1+h_0(k))^{-2} \left(\sum_{a=k-h_0(k)+1}^k \Lip_a(K)/s\right)^2
  \leq C s^{-2} \sum_k M(k)^2,
  \end{multline*}
where $M(k)$ is the maximal function associated to $\Lip_a(K)$,
defined in~\eqref{def_Mk}. By Theorem~\ref{thm:maximal}, it satisfies
$\sum_k M(k)^2 \leq C \sum \Lip_a(K)^2$. Hence, we have proved that
the weak $L^2$ norm of $v$ is bounded by $C \left(\sum
\Lip_a(K)^2\right)^{1/2}$.

Using the bounds on the weak $L^2$ norm of $v$ and on its $L^\infty$
norm, one deduces a bound on its strong $L^2$ norm as
in~\eqref{eq:deduce_strong_from_weak}, and on its strong $L^p$ norm
for $p>2$ as in~\eqref{eq:absvp}. These bounds read:
  \begin{multline*}
  \sum_k \sum_h \mu(\phi=h+1) \left(\sum_{a=k-h+1}^k \Lip_a(K)\right)^2
  \\
  \leq C \left(\sum_{a\in \Z} \Lip_a(K)^2\right)
    \left[1+\log\left(\sum_{a\in \Z} \Lip_a(K)\right)-\log\left(\sum_{a\in \Z} \Lip_a(K)^2\right)^{1/2}\right]
  \end{multline*}
and for $p>2$
  \begin{equation*}
  \sum_k \sum_h \mu(\phi=h+1) \left(\sum_{a=k-h+1}^k \Lip_a(K)\right)^p
  \leq C \left(\sum_{a\in \Z} \Lip_a(K)^2\right) \left(\sum_{a\in \Z} \Lip_a(K)\right)^{p-2}.
  \end{equation*}

For $p=2$, we deduce directly that the contribution of the second
line of~\eqref{eq:q=2_D_k} to Burkholder-Rosenthal inequality is
bounded as in the statement of the theorem.

For $p>2$, we also obtain that the contribution of this line, for
$r=p$, is bounded as desired. It remains to check the contribution of
this line with $r=2$. Writing $u_a=\Lip_a(K)$, we should prove that
  \begin{equation*}
  \left(\sum_{a\in \Z} u_a^2\right)^{p/2}
    \left[1+\log\left(\sum_{a\in \Z} u_a\right)-\log\left(\sum_{a\in \Z} u_a^2\right)^{1/2}\right]^{p/2}
  \leq C \left(\sum_{a\in \Z} u_a^2\right) \left(\sum_{a\in \Z} u_a\right)^{p-2}.
  \end{equation*}
Since this equation is homogeneous, it suffices to prove it when
$\sum u_a^2=1$. In this case, writing $x = \sum u_a \geq 1$, it
reduces to the inequality $(1+\log x)^{p/2} \leq C x^{p-2}$, which is
trivial on $[1, \infty)$.
\qed

\appendix
\section{Speed of convergence to stable laws}
\label{sec:lower_stable}

In this appendix, our goal is to prove
Proposition~\ref{prop:lower_stable}. To do so, we estimate the
speed of convergence of the Birkhoff sums to the stable law, first on
the basis $\Delta_0$ of the tower using the Nagaev-Guivarc'h spectral
method. Then, we induce back those estimates to the whole tower.
Those ideas are classical: the first step comes
from~\cite{aaronson_denker}, the second step
from~\cite{melbourne_torok} (see~\cite{gouezel_lausanne} for a
general explanation of the method). However, since we want
quantitative estimates, we need to go beyond the results of these
articles.

The standing assumptions are those of
Proposition~\ref{prop:lower_stable}: $(\Delta,T)$ is a Young tower
with $\tail_n = C n^{-q}+O(n^{-q-\epsilon})$, for some $q\in (1,2)$
and some $\epsilon>0$. Without loss of generality, we can assume
$q+\epsilon<2$.

Let $Y=\Delta_0$ be the basis of the Young tower. We denote by
$T_Y:Y\to Y$ the induced map on the basis, by $\mu_Y =
\mu\restr{Y}/\mu(Y)$ the induced probability measure, by $S_n^Y$ the
Birkhoff sums for $T_Y$, and by $\phi:Y\to \N$ the first return time
to $Y$.

We define a function $f$ on the tower, by $f=1 - 1_Y /\mu(Y)$, so
that $\int f\dd\mu=0$. The induced function on the basis of the tower
is by definition
  \begin{equation*}
  f_Y(x) = \sum_{k=0}^{\phi(x)-1} f(T^k x) = \phi(x) -1_Y/\mu(Y).
  \end{equation*}
Denote by $L$ the transfer operator associated to $T_Y$, and define a
family of perturbed transfer operators $L_t(u)=L(e^{\ic t f_Y} u)$.
Their interest is that
  \begin{equation}
  \label{eq:charac}
  \int_Y e^{\ic t S_n^Y f_Y} \dd\mu_Y= \int_Y L_t^n 1 \dd\mu_Y.
  \end{equation}
Hence, spectral properties of $L_t$ make it possible to
understand the characteristic function of $S_n^Y f_Y$, and therefore
its closeness to the limiting stable law.

\begin{lem}
The family of operators $t\mapsto L_t$ is $C^1$.
\end{lem}
\begin{proof}
We omit the standard argument which shows in fact that the family is
$C^q$, see for instance~\cite[Theorem 2.4]{aaronson_denker}.
\end{proof}

The unperturbed operator $L=L_0$ has a simple eigenvalue at $1$, and
the rest of its spectrum is contained in a disk of strictly smaller
radius. This spectral description persists for small $t$,
see~\cite{kato_pe}. Denote by $\lambda_t$ the leading eigenvalue of
$L_t$, by $\Pi_t$ the corresponding (one-dimensional) spectral
projection, and by $Q_t=L_t -\lambda_t \Pi_t$ the part of $L_t$
corresponding to the rest of the spectrum. All those quantities
depend in a $C^1$ way on $t$, by the previous proposition. Moreover,
for small $t$, we have
  \begin{equation}
  \label{eq:spectral_expansion}
  L_t^n = \lambda_t^n \Pi_t + Q_t^n,\quad \text{with } \norm{Q_t^n} \leq Cr^n,
  \end{equation}
for some fixed $r<1$. The main contribution in this equation comes
from the perturbed eigenvalue $\lambda_t$.

\begin{lem}
\label{lem:lambdat}
We have for small $t>0$
  \begin{equation*}
  \lambda_t= 1 + c t^q + O(t^{q+\epsilon}),
  \end{equation*}
where $c$ is a complex number with $\Re c <0$.
\end{lem}
\begin{proof}
Let $\xi_t$ denote the $C^1$ family of eigenfunctions of $L_t$ for
the eigenvalue $\lambda_t$, normalized so that $\int_Y \xi_t
\dd\mu_Y=1$. In particular, $\xi_0\equiv 1$.

Now
  \begin{align*}
  \lambda_t & =\int_Y \lambda_t \xi_t\dd\mu_Y
  =\int_Y L_t\xi_t\dd\mu_Y
  =\int_Y L_t1\dd\mu_Y + \int_Y (L_t-L)(\xi_t-1)\dd\mu_Y
  \\ & =\int_Y e^{\ic tf_Y}\dd\mu_Y+ O(t^2)
  \\ & =1 + \int_Y (e^{\ic tf_Y}-1-\ic tf_Y)\dd\mu_Y+ O(t^2).
  \end{align*}

Let $G(s)=\mu_Y(f_Y<s)$ denote the distribution function of $f_Y$. It
vanishes for $s<-M$, for $M=1-1/\mu(Y)$. The asymptotics of the tails
of the return time yield $1-G(s)=c_1 s^{-q}+A(s)$ where
$A(s)=O(s^{-q-\epsilon})$. Hence
  \begin{align*}
  \int_Y (e^{\ic t f_Y}-1-\ic t f_Y)\dd\mu_Y & =
  \int_{-M}^\infty (e^{\ic ts}-1-\ic ts)\dd G(s)  \\ & =
  -\int_{-M}^\infty (e^{\ic ts}-1-\ic ts)\dd(1-G(s))
  \\ & =(e^{-\ic tM}-1+\ic tM)+\ic t \int_{-M}^\infty (e^{\ic ts}-1)(1-G(s))\dd s
  \\ & =\ic t \int_1^\infty (e^{\ic ts}-1)(1-G(s))\dd s
  +O(t^2)
  \\ & =\ic \int_t^\infty (e^{\ic\sigma}-1)(1-G(\sigma/t))\dd\sigma + O(t^2) \\ & =
  \ic c_1t^ q \int_t^\infty (e^{\ic\sigma}-1)\sigma^{- q }\dd\sigma+
  \ic \int_t^\infty (e^{\ic\sigma}-1)A(\sigma/t)\dd\sigma + O(t^2) \\ & =
  c_2t^ q +E_1+E_2 + O(t^2),
  \end{align*}
where
  \begin{align*}
  c_2 &= \ic c_1\int_0^\infty (e^{\ic\sigma}-1)\sigma^{- q }\dd\sigma, \quad
  E_1  = -\ic c_1t^{ q }\int_0^t (e^{\ic\sigma}-1)\sigma^{- q }\dd\sigma, \\
  E_2 &=\ic \int_t^\infty (e^{\ic\sigma}-1) A(\sigma/t)\dd\sigma.
  \end{align*}
Note that $c_2$ is well-defined since $ q \in(1,2)$. Also,
$\abs{E_1}\le c_1t^{ q }\int_0^t\sigma^{-( q -1)}\dd\sigma=O(t^2)$.
There is a constant $C>0$ such that $\abs{A(s)}\le Cs^{-q-\epsilon}$
for $s\ge1$.  Hence
  \begin{equation*}
  \abs{E_2}\le Ct^{q+\epsilon} \int_t^\infty \abs{e^{\ic \sigma}-1} \sigma^{-q-\epsilon}
  \dd\sigma
  \leq Ct^{q+\epsilon} \int_0^\infty \min( \sigma^{1-q-\epsilon}, \sigma^{-q-\epsilon})\dd\sigma=O(t^{q+\epsilon}),
  \end{equation*}
where the integral is finite since $q+\epsilon \in (1,2)$.
\end{proof}

Let $Z_Y$ be the real probability distribution whose characteristic
function is given for $t>0$ by
  \begin{equation*}
  \E(e^{\ic t Z_Y})=e^{\ic c t^q},
  \end{equation*}
where $c$ is given by Lemma~\ref{lem:lambdat}. It is a totally
asymmetric stable law of index $q$. We can now estimate the speed of
convergence of $S_n^Y f_Y$ to $Z_Y$:
\begin{prop}
\label{prop:BE_Y}
There exists $C>0$ such that for any $n>0$ and for any $s\in \R$,
  \begin{equation*}
  \abs*{ \mu_Y\{ x\st S_n^Y f_Y(x)/n^{1/q} > s\} - \Pbb(Z_Y > s)} \leq C n^{-\epsilon/q}.
  \end{equation*}
\end{prop}
In particular, we recover the (already known) convergence of $S_n^Y
f_Y/n^{1/q}$ to $Z_Y$, the novelty being the control on the speed of
convergence. Below, in Proposition~\ref{prop:BE_Delta} and
Theorem~\ref{thm:BE_Delta}, we also recover known stable limits,
with additional controls on the speed of convergence.
\begin{proof}
The quantity to estimate is the $L^\infty$-norm of the difference
between the distribution functions of $S_n^Y f_Y/n^{1/q}$ and $Z_Y$.
Berry-Esseen's lemma (see for instance~\cite[Lemma
XVI.3.2]{feller_2}) ensures that, for any $M>0$, this quantity is
bounded by
  \begin{equation}
  \label{eq:BE}
  C\int_{0}^M \frac{\abs*{\phi_n(t) -\psi(t)}}{t}\dd t + \frac{C}{M}
  \end{equation}
where $C$ is a universal constant, and $\phi_n$ and $\psi$ denote
respectively the characteristic functions of $S_n^Y f_Y/n^{1/q}$ and
$Z_Y$. We estimate this integral, taking $M=M_n=\alpha_0
n^{1/q}$ for some suitably small $\alpha_0$.

First, for $t<1/n$, we have
  \begin{equation*}
  \abs{\phi_n(t) -1} = \abs{\E(e^{\ic t S_n^Y f_Y/n^{1/q}}-1)}
  \leq t \int \abs{S_n^Y f_Y}\dd\mu_Y/n^{1/q} \leq Ct n^{1-1/q}.
  \end{equation*}
In the same way, $\abs{\psi(t)-1} \leq Ct^q$. Hence,
  \begin{equation*}
  \int_{0}^{1/n} \frac{\abs*{\phi_n(t) -\psi(t)}}{t}
  \leq C n^{-1/q}+ C n^{-q} \leq C n^{-1/q}.
  \end{equation*}

Now, we turn to the interval $t\in [1/n, M_n]$. Combining the
formula~\eqref{eq:charac} for $\phi_n$ and the spectral
expansion~\eqref{eq:spectral_expansion} of $L_t^n$, we get
  \begin{equation*}
  \phi_n(t) = \lambda(t/n^{1/q})^n u(t/n^{1/q}) + r_n(t/n^{1/q}),
  \end{equation*}
where $r_n$ is exponentially small, $u$ is a $C^1$ function close to
$0$ and the asymptotic expansion of $\lambda$ is given in
Lemma~\ref{lem:lambdat}. The contribution of $r_n$ to the
integral~\eqref{eq:BE} is exponentially small (this is why we had to
discard the interval $[0, 1/n]$). We can write $\lambda(s)=e^{c s^q +
B(s)}$ where $B(s)=O(s^{q+\epsilon})$, by Lemma~\ref{lem:lambdat}.
Hence,
  \begin{equation*}
  \lambda(t/n^{1/q})^n = e^{n(c t^q/n + B(t/n^{1/q}))} = e^{ct^q} e^{n B(t/n^{1/q})}
  = \psi(t) e^{n B(t/n^{1/q})}.
  \end{equation*}
The remaining part of the integral~\eqref{eq:BE} can be written as
  \begin{multline*}
  \int_{1/n}^{M_n} \frac{\abs{ \psi(t) e^{n B(t/n^{1/q})} u(t/n^{1/q}) -\psi(t)}}{t} \dd t
  \\
  \leq \int_0^{M_n} \abs*{\psi(t) e^{n B(t/n^{1/q})}} \frac{ \abs{u(t/n^{1/q}) -1}}{t}\dd t
  + \int_0^{M_n} \abs{\psi(t)} \frac{\abs{e^{n B(t/n^{1/q})} -1}}{t}\dd t
  \eqqcolon I_1 + I_2.
  \end{multline*}
In $I_1$, we have
  \begin{equation*}
  \frac{n \abs{B(t/n^{1/q})}}{t^q} \leq C (t/n^{1/q})^{q+\epsilon-q}
  \leq C \alpha_0^{\epsilon}.
  \end{equation*}
Hence, $\abs*{\psi(t) e^{n B(t/n^{1/q})}} \leq e^{\Re(c) t^q} e^{C
\alpha_0^{\epsilon} t^q}$. If $\alpha_0$ is small enough, this is
bounded by $e^{-a t^q}$, for some $a>0$. Since the function $u$ is
$C^1$ with $u(0)=1$, it follows that
  \begin{equation*}
  I_1 \leq C\int_0^{M_n} e^{-a t^q} n^{-1/q} \dd t \leq C n^{-1/q}.
  \end{equation*}
Finally, in $I_2$, we use the inequality $\abs{e^s-1} \leq
\abs{s}e^{\abs{s}}$, to get a bound
  \begin{equation*}
  I_2 \leq \int_0^{M_n} \abs*{\psi(t) e^{n \abs{B(t/n^{1/q})}}}  \frac{n \abs{B(t/n^{1/q})}}{t}\dd t.
  \end{equation*}
As above, the factor $\abs*{\psi(t) e^{n \abs{B(t/n^{1/q})}}}$ is
bounded by $e^{-a t^q}$. Moreover, the second factor is bounded by
$t^{q+\epsilon-1} n^{-\epsilon/q}$. This gives $I_2 \leq C
n^{-\epsilon/q}$.

Finally, we obtain a bound for~\eqref{eq:BE} of the form $C n^{-1/q}
+ C n^{-\epsilon/q}$, which is bounded by $C n^{-\epsilon/q}$ as
$\epsilon< 2-q<1$.
\end{proof}

We can then lift the above bound to the original Birkhoff sums $S_n
f$. Let $Z=\mu(Y)^{1/q}Z_Y$, it is again a (completely asymmetric)
stable law of index $q$.
\begin{prop}
\label{prop:BE_Delta}
Let $\delta=\min((q-1)/(1+2q^2), \epsilon/q) > 0$. There exists $C>0$
such that for any $n>0$ and for any $s\in \R$,
  \begin{equation*}
  \abs*{ \mu_Y\{ x\st S_n f(x)/n^{1/q} > s\} - \Pbb(Z > s)} \leq C n^{-\delta}.
  \end{equation*}
\end{prop}
\begin{proof}
For $x\in Y$, the Birkhoff sums $S_n f(x)$ and $S^Y_{n\mu(Y)} f_Y(x)$
should be close (since a return to $Y$ takes on average $1/\mu(Y)$
iterates of $T$, both sums involve roughly the same number of
iterations of $T$), and we know that $S^Y_{n\mu(Y)}
f_Y(x)/(n\mu(Y))^{1/q}$ is close to $Z_Y$ in distribution. (We write
$n\mu(Y)$ instead of its integer part for notational simplicity.) The
result follows if we can show that the different errors are suitably
small.

Define a function $H$ on $\Delta$ as follows: if $x$ is at height $i$
(i.e., it belongs to $\Delta_{\alpha,i}$ for some $\alpha$), let $\pi
x=T^{-i}x$ be its unique preimage in the basis, and let
$H(x)=\sum_{j=0}^{i-1}f(T^j \pi x)$. Let $N(n,x)$ denote the number
of returns to $Y$ of a point $x\in Y$ before time $n$. We get $S_n
f(x) = S_{N(n,x)}^Y f_Y(x) + H(T^n x)$. We expect $N(n,x)$ to be
close to $n \mu(Y)$, hence we decompose further as
  \begin{align*}
  S_n f(x) & = S^Y_{n\mu(Y)} f_Y(x) + (S_{N(n,x)}^Y f_Y(x) - S^Y_{n\mu(Y)} f_Y(x)) + H(T^n x)
  \\&
  = S^Y_{n\mu(Y)} f_Y(x) + E_n(x) + F_n(x).
  \end{align*}
Suppose that, for $u_n=n^{-\delta}$ for some $\delta\in (0,
\epsilon/q]$, we have
  \begin{equation}
  \label{eq:bounds_Prok}
  \mu_Y\{ \abs{E_n}/n^{1/q} > u_n\} \leq Cu_n,\quad \mu_Y\{\abs{F_n}/n^{1/q} > u_n\} \leq Cu_n.
  \end{equation}

We deduce from the above equation that
  \begin{equation*}
  \mu_Y\{S_n f/ n^{1/q} > s\} \leq \mu_Y\{S_{n\mu(Y)}^Y f_Y / n^{1/q} > s-2 u_n\} + 2Cu_n.
  \end{equation*}
By Proposition~\ref{prop:BE_Y}, this is bounded by
  \begin{equation*}
  \Pbb(\mu(Y)^{1/q} Z_Y > s-2u_n) + Cn^{-\epsilon/q} + 2Cu_n.
  \end{equation*}
As $Z_Y$ has a bounded density, the probability that
$\mu(Y)^{1/q}Z_Y$ belongs to the interval $[s-2u_n, s)$ is bounded by
$C u_n$. Finally, we obtain
  \begin{equation*}
  \mu_Y\{S_n f/ n^{1/q} > s\} \leq \Pbb(\mu(Y)^{1/q} Z_Y > s) + C n^{-\epsilon/q} + Cu_n.
  \end{equation*}
The lower bound is similar, and we obtain the conclusion of the
proposition.

It remains to prove~\eqref{eq:bounds_Prok}. We first deal with the
bound involving $F_n$. We have
  \begin{equation*}
  \mu(F_n \geq u_n n^{1/q}) =\mu(H\circ T^n \geq u_n n^{1/q}) =\mu(H \geq u_n n^{1/q}).
  \end{equation*}
The function $H$ can only be $\geq A$ on the set of points with
height at least $A$. The set of points with height $i$ has measure
$\tail_{i+1} \sim C i^{-q}$, hence $\mu(H\geq A) \leq C A^{-q+1}$. We
get
  \begin{equation*}
  \mu(F_n \geq u_n n^{1/q}) \leq C (u_n n^{1/q})^{-(q-1)}.
  \end{equation*}
This is  bounded by $Cu_n$ if $u_n=n^{-\delta}$ with $\delta\leq
(q-1)/q^2$.

We turn to $E_n$. Let $M_n=n^{r}$, for some $r \in (1/q, 1)$. We have
  \begin{equation*}
  \{E_n \geq u_n n^{1/q}\} \subset
   \{E_n(x) \geq u_n n^{1/q},\ \abs{N(n,x)-n\mu(Y)}< M_n \}
   \cup \{ \abs{N(n,x)-n\mu(Y)}\geq M_n \}.
   \end{equation*}
In the first set, as $S^Y_{N(n,x)}f_Y(x)$ and $S^Y_{n\mu(Y)}f_Y(x)$
are separated by $u_n n^{1/q}$, one of them is distant from
$S^Y_{n\mu(Y)-M_n} f_Y(x)$ by at least $u_n n^{1/q}/2$. Hence, the
first set is included in
  \begin{equation*}
  \left\{ \max_{0\leq k \leq 2M_n} \abs{S^Y_{n\mu(Y)-M_n+k} f_Y(x)
      -  S^Y_{n\mu(Y)-M_n} f_Y(x)} \geq u_n n^{1/q}/2\right\}.
  \end{equation*}
By the invariance of the measure $\mu_Y$ under $T_Y$, the measure of
this set is
  \begin{equation*}
  \mu_Y\left\{ \max_{0\leq k \leq 2M_n} \abs{S^Y_k f_Y} \geq u_n n^{1/q}/2\right\}.
  \end{equation*}
The sequence $S_i^Y f_Y/ i^{1/q}$ converges in distribution, but more
is true: It follows from~\cite[Lemma 7.1 and proof of Theorem
2.10]{gouezel_chazottes_TCLps} that this sequence remains bounded in
$L^1$, and that the weak $L^1$ norm of the corresponding maxima also
remain bounded. Hence, the above equation is bounded by $C
M_n^{1/q}/(u_n n^{1/q})$. This is bounded by $Cu_n$ if
$u_n=n^{-\delta}$ with $\delta \leq (1-r)/(2q)$.

Finally, if $\abs{N(n,x)-n\mu(Y)}\geq M_n$, then either $N(n,x) \geq
n\mu(Y)+M_n$, or $N(n,x) \leq n\mu(Y)-M_n$. In the first case,
$S^Y_{n\mu(Y)+M_n}\phi_Y \leq n$, i.e., $S^Y_{n\mu(Y)+M_n} f_Y \leq
-M_n/\mu(Y)$. By Proposition~\ref{prop:BE_Y}, this can only happen
with probability $\Pbb(Z_Y \leq -c M_n \mu(Y)/n^{1/q}) + C
n^{-\epsilon/q}$. The stable law $Z_Y$ has tails of order $q$, i.e.,
$\Pbb(\abs{Z_Y}>s)\leq Cs^{-q}$. Hence, this is bounded by
$u_n=n^{-\delta}$ if $\delta \leq \min(q(r-1/q), \epsilon/q)$. The
second case is handled similarly.

We have proved that, if $\delta$ is small enough,
then~\eqref{eq:bounds_Prok} holds. More specifically, we can choose
$r$ so that $q(r-1/q) = (1-r)/2q$, i.e., $r= (1+2q)/(1+2q^2)$. The
resulting constraints on $\delta$ are
  \begin{equation*}
  \delta \leq \min( (q-1)/q^2, (q-1)/(1+2q^2), \epsilon/q).
  \end{equation*}
The first constraint can be removed since it is implied by the second
one.
\end{proof}

We can now conclude the proof of Proposition~\ref{prop:lower_stable}.
The probability distribution $Z$ has heavy tails, since it is a
stable law of index $q$: there exists $c>0$ such that, for all $s\geq
1$, we have $\Pbb(Z>s) \geq c s^{-q}$. It follows from
Proposition~\ref{prop:BE_Delta} that $\mu_Y\{S_n f/n^{1/q} >s\} \geq
c s^{-q} - C n^{-\delta}$. This is $\geq c s^{-q}/2$ if $C
n^{-\delta} \leq c s^{-q}/2$, which holds for $s \in [1, n^r]$ if
$r<\delta/q$ and $n$ is large enough. In this range, it follows that
$\mu\{S_n f/n^{1/q} > s\} \geq c' s^{-q}$.

Using~\eqref{eq:deduce_moments} for the first equality, we have
  \begin{equation*}
  \int \abs{S_n f/n^{1/q}}^q \dd\mu
  =q\int_{s=0}^\infty s^{q-1} \mu\left\{\abs{S_n f/n^{1/q}}\geq s\right\} \dd s
  \geq q\int_{s=1}^{n^r} s^{q-1} c' s^{-q} \dd s
  =c'q r \log n.
  \end{equation*}
This is the desired lower bound. \qed

\medskip

One can also deduce from Proposition~\ref{prop:BE_Delta} a speed of
convergence towards the stable law $Z$ on the whole space
$(\Delta,\mu)$. Although this is not needed for
Proposition~\ref{prop:lower_stable}, we include it for completeness:
\begin{thm}
\label{thm:BE_Delta}
Let $\delta=\min( (q-1)/(1+2q^2), \epsilon/q)$. There exists $C>0$
such that for any $n>0$ and for any $s\in \R$,
  \begin{equation*}
  \abs*{ \mu\{ x\st S_n f(x)/n^{1/q} > s\} - \Pbb(Z > s)} \leq C n^{-\delta}.
  \end{equation*}
\end{thm}
\begin{proof}
Consider a set $\Delta_{\alpha,i}$, with its renormalized probability
measure $\mu_{\alpha,i} =
\mu\restr{\Delta_{\alpha,i}}/\mu(\Delta_{\alpha,i})$. This measure is
sent by $T^{h_\alpha-i}_*$ to a measure on $Y$, which is equivalent
to $\mu_Y$, with a density bounded from above and from below, and
with uniformly bounded Lipschitz constant.
Proposition~\ref{prop:BE_Y} still works for this measure, with
uniform constants, since all we need to apply the spectral argument
is that the density is Lipschitz. It follows that
Proposition~\ref{prop:BE_Delta} also works for these measures. Adding
the additional error coming from the $h_\alpha-i$ first steps needed
to reach $Y$ (bounded by $(h_\alpha-i)/n^{1/q}$), we deduce: for
$n>h_\alpha-i$,
  \begin{equation*}
  \abs*{\mu_{\alpha,i}\{x\in \Delta_{\alpha,i} \st S_n f(x)/n^{1/q} > s\}
        - \Pbb(Z>s) }
  \leq C (n- (h_\alpha - i))^{-\delta} + C (h_\alpha-i)/n^{1/q}.
  \end{equation*}
Let $\Lambda_k$ denote the set of points in $\Delta$ that enter $Y$
after exactly $k$ steps. Multiplying the above inequality by
$\mu(\Delta_{\alpha,i})$ and summing over $(\alpha,i)$, we obtain:
  \begin{multline*}
  \abs*{\mu\{x\in \Delta \st S_n f(x)/n^{1/q} > s\} - \Pbb(Z>s)}
  \\
  \leq C \sum_{k<n} \mu(\Lambda_k) \min((n-k)^{-\delta} + k/n^{1/q}, 1) + \sum_{k\geq n} \mu(\Lambda_k).
  \end{multline*}
We have $\mu(\Lambda_k) = \tail_{k+1} \sim C k^{-q}$. Splitting the
above sum into $k\leq n^{1/q}$ and $k>n^{1/q}$, we get the bound
  \begin{equation*}
  C\sum_{k\leq n^{1/q}} k^{-q} (n^{-\delta} + k/n^{1/q}) + C\sum_{k> n^{1/q}} k^{-q}
  \leq C n^{-\delta} + C n^{(2-q)/q} / n^{1/q} + C n^{(1-q)/q}.
  \end{equation*}
This is bounded by $Cn^{-\delta'}$ for $\delta'=\min(\delta,
1-1/q)=\delta$.
\end{proof}

\bibliography{biblio}

\def\cprime{$'$}
\providecommand{\bysame}{\leavevmode\hbox to3em{\hrulefill}\thinspace}
\providecommand{\MR}{\relax\ifhmode\unskip\space\fi MR }
\providecommand{\MRhref}[2]{%
  \href{http://www.ams.org/mathscinet-getitem?mr=#1}{#2}
}
\providecommand{\href}[2]{#2}
\begin{thebibliography}{CESFZ08}

\bibitem[AD01]{aaronson_denker}
Jon Aaronson and Manfred Denker, \emph{Local limit theorems for partial sums of
  stationary sequences generated by {G}ibbs-{M}arkov maps}, Stoch. Dyn.
  \textbf{1} (2001), 193--237. \MR{MR1840194}

\bibitem[AHO03]{Armstead_etal03}
Douglas~N. Armstead, Brian~R. Hunt, and Edward Ott, \emph{Anomalous diffusion
  in infinite horizon billiards}, Phys. Rev. E (3) \textbf{67} (2003), 021110,
  7 pages. \MR{MR1974617}

\bibitem[ASW11]{astashkin_weak_lp_vonbahr}
Sergey Astashkin, Fedor Sukochev, and Chin~Pin Wong, \emph{Disjointification of
  martingale differences and conditionally independent random variables with
  some applications}, Studia Math. \textbf{205} (2011), 171--200.
  \MR{MR2824894}

\bibitem[BCD11]{BalintChernovDolgopyat11}
P{\'e}ter B{\'a}lint, Nikolai Chernov, and Dmitry Dolgopyat, \emph{Limit
  theorems for dispersing billiards with cusps}, Comm. Math. Phys. \textbf{308}
  (2011), 479--510. \MR{MR2851150}

\bibitem[BCD13]{Chernov_preprint}
\bysame, \emph{Convergence of moments for dispersing billiards with cusps},
  Preprint, 2013.

\bibitem[BG06]{balint_gouezel}
P{\'e}ter B{\'a}lint and S{\'e}bastien Gou{\"e}zel, \emph{Limit theorems in the
  stadium billiard}, Comm. Math. Phys. \textbf{263} (2006), 461--512.
  \MR{MR2207652}

\bibitem[Bra94]{braverman_rearrangement}
Michael~Sh. Braverman, \emph{Independent random variables and rearrangement
  invariant spaces}, London Mathematical Society Lecture Note Series, vol. 194,
  Cambridge University Press, Cambridge, 1994. \MR{MR1303591}

\bibitem[Bur73]{burkholder}
Donald~L. Burkholder, \emph{Distribution function inequalities for
  martingales}, Ann. Probability \textbf{1} (1973), 19--42. \MR{MR0365692}

\bibitem[CESFZ08]{Courbage_etal08}
M.~Courbage, M.~Edelman, S.~M. Saberi~Fathi, and G.~M. Zaslavsky, \emph{Problem
  of transport in billiards with infinite horizon}, Phys. Rev. E (3)
  \textbf{77} (2008), 036203, 5 pages. \MR{MR2495438}

\bibitem[CG07]{gouezel_chazottes_TCLps}
Jean-Ren{\'e} Chazottes and S{\'e}bastien Gou{\"e}zel, \emph{On almost-sure
  versions of classical limit theorems for dynamical systems}, Probab. Theory
  Related Fields \textbf{138} (2007), 195--234. \MR{MR2288069}

\bibitem[CG12]{gouezel_chazottes_concentration}
\bysame, \emph{Optimal concentration inequalities for dynamical systems}, Comm.
  Math. Phys. \textbf{316} (2012), 843--889. \MR{MR2993935}

\bibitem[DM14]{dedecker_merlevede_moment}
J{\'e}r{\^o}me Dedecker and Florence Merlev{\`e}de, \emph{Moment bounds for
  dependent sequences in smooth banach spaces}, Preprint, 2014.

\bibitem[Fel66]{feller_2}
William Feller, \emph{An introduction to probability theory and its
  applications. {V}ol. {II}}, John Wiley \& Sons Inc., New York, 1966.
  \MR{MR0210154}

\bibitem[Gou04a]{gouezel_stable}
S{\'e}bastien Gou{\"e}zel, \emph{Central limit theorem and stable laws for
  intermittent maps}, Probab. Theory Related Fields \textbf{128} (2004),
  82--122. \MR{MR2027296}

\bibitem[Gou04b]{gouezel_decay}
\bysame, \emph{Sharp polynomial estimates for the decay of correlations},
  Israel J. Math. \textbf{139} (2004), 29--65. \MR{MR2041223}

\bibitem[Gou04c]{gouezel_these}
\bysame, \emph{Vitesse de d{\'e}corr{\'e}lation et th{\'e}or{\`e}mes limites
  pour les applications non uniform{\'e}ment dilatantes}, Ph.D. thesis,
  Universit{\'e} Paris Sud, 2004.

\bibitem[Gou13]{gouezel_lausanne}
\bysame, \emph{Limit theorems in dynamical systems using the spectral method},
  preprint, 2013.

\bibitem[Kat66]{kato_pe}
Tosio Kato, \emph{Perturbation theory for linear operators}, Die Grundlehren
  der mathematischen Wissenschaften, Band 132, Springer-Verlag New York, Inc.,
  New York, 1966. \MR{MR0203473}

\bibitem[LSV99]{liverani_saussol_vaienti}
Carlangelo Liverani, Beno{\^{\i}}t Saussol, and Sandro Vaienti, \emph{A
  probabilistic approach to intermittency}, Ergodic Theory Dynam. Systems
  \textbf{19} (1999), 671--685. \MR{MR1695915}

\bibitem[Mel09]{M09b}
Ian Melbourne, \emph{Large and moderate deviations for slowly mixing dynamical
  systems}, Proc. Amer. Math. Soc. \textbf{137} (2009), 1735--1741.
  \MR{MR2470832}

\bibitem[MN08]{melbourne_nicol_large_deviations}
Ian Melbourne and Matthew Nicol, \emph{Large deviations for nonuniformly
  hyperbolic systems}, Trans. Amer. Math. Soc. \textbf{360} (2008), 6661--6676.
  \MR{MR2434305}

\bibitem[MT04]{melbourne_torok}
Ian Melbourne and Andrew T{\"o}r{\"o}k, \emph{Statistical limit theorems for
  suspension flows}, Israel J. Math. \textbf{144} (2004), 191--209.
  \MR{MR2121540}

\bibitem[MT12a]{melbourne_terhesiu}
Ian Melbourne and Dalia Terhesiu, \emph{Operator renewal theory and mixing
  rates for dynamical systems with infinite measure}, Invent. Math.
  \textbf{189} (2012), 61--110. \MR{MR2929083}

\bibitem[MT12b]{MTorok12}
Ian Melbourne and Andrei T{\"o}r{\"o}k, \emph{Convergence of moments for
  {A}xiom {A} and non-uniformly hyperbolic flows}, Ergodic Theory Dynam.
  Systems \textbf{32} (2012), 1091--1100. \MR{MR2995657}

\bibitem[PM80]{pomeau_manneville}
Yves Pomeau and Paul Manneville, \emph{Intermittent transition to turbulence in
  dissipative dynamical systems}, Comm. Math. Phys. \textbf{74} (1980),
  189--197. \MR{MR576270}

\bibitem[Sar02]{sarig_decay}
Omri Sarig, \emph{Subexponential decay of correlations}, Invent. Math.
  \textbf{150} (2002), 629--653. \MR{MR1946554}

\bibitem[SV07]{szasz_varju_infinite}
Domokos Sz{\'a}sz and Tam{\'a}s Varj{\'u}, \emph{Limit laws and recurrence for
  the planar {L}orentz process with infinite horizon}, J. Stat. Phys.
  \textbf{129} (2007), 59--80. \MR{MR2349520}

\bibitem[SW71]{stein_weiss_fourier}
Elias~M. Stein and Guido Weiss, \emph{Introduction to {F}ourier analysis on
  {E}uclidean spaces}, Princeton University Press, Princeton, N.J., 1971,
  Princeton Mathematical Series, No. 32. \MR{MR0304972}

\bibitem[vBE65]{vonbahr_esseen}
Bengt von Bahr and Carl-Gustav Esseen, \emph{Inequalities for the {$r$}th
  absolute moment of a sum of random variables, {$1\leq r\leq 2$}}, Ann. Math.
  Statist \textbf{36} (1965), 299--303. \MR{MR0170407}

\bibitem[You98]{lsyoung_annals}
Lai-Sang Young, \emph{Statistical properties of dynamical systems with some
  hyperbolicity}, Ann. of Math. (2) \textbf{147} (1998), 585--650.
  \MR{MR1637655}

\bibitem[You99]{lsyoung_recurrence}
\bysame, \emph{Recurrence times and rates of mixing}, Israel J. Math.
  \textbf{110} (1999), 153--188. \MR{MR1750438}

\end{thebibliography}
\bibliographystyle{amsalpha}
\end{document}